\documentclass[11pt]{article}
\def\softl{{l\kern-0.3ex\raise0.1ex\hbox{'}\kern-0.10ex}}
\usepackage[left=2cm,right=2cm,top=3cm,bottom=3cm]{geometry}

\usepackage{amsmath, amssymb, latexsym, amsthm}
\usepackage{fullpage}
\usepackage{pstricks,pst-node,pst-pdf}
\usepackage{color}
\usepackage{tikz}
\usepackage{graphicx}
\usepackage{stmaryrd,amsbsy,enumerate}
\usepackage{hyperref}
\usetikzlibrary{calc}
\usepackage{enumitem}

\def\comment#1{}
\let\comment=\footnote 
\def\figscale{0.7}

\newcommand{\vc}[1]{\ensuremath{\vcenter{\hbox{#1}}}}
\newcommand{\figV}{\vc{\includegraphics[page=1,scale=\figscale]{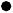}} }
\newcommand{\figN}{\vc{\includegraphics[page=2,scale=\figscale]{EiC-fig}} }
\newcommand{\figB}{\vc{\includegraphics[page=3,scale=\figscale]{EiC-fig}} }
\newcommand{\figR}{\vc{\includegraphics[page=4,scale=\figscale]{EiC-fig}} }
\newcommand{\figTRG}{\vc{\includegraphics[page=5,scale=\figscale]{EiC-fig}} }
\newcommand{\figEMP}{\vc{\includegraphics[page=20,scale=\figscale]{EiC-fig}} }
\newcommand{\figCOCh}{\vc{\includegraphics[page=21,scale=\figscale]{EiC-fig}} }

\newcommand{\cupdot}{\mathbin{\mathaccent\cdot\cup}}

\newtheorem{theorem}{Theorem} 
\newtheorem{theorem*}{Theorem} 
\newtheorem{thm}[theorem]{Theorem} 
\newtheorem{prop}[theorem]{Proposition}

\newtheorem{cor}[theorem]{Corollary}
\newtheorem{lem}[theorem]{Lemma}
\newtheorem{conj}[theorem]{Conjecture}
\newtheorem{obs}[theorem]{Observation}

\newtheorem{claim}[theorem]{Claim}

\theoremstyle{definition}

\newtheorem{construction}{Construction}

\theoremstyle{remark}

\numberwithin{theorem}{section}

\let\oldlfloor\lfloor
\let\oldrfloor\rfloor
\let\oldlceil\lceil
\let\oldrceil\rceil
\renewcommand{\lfloor}{\left\oldlfloor}
\renewcommand{\rfloor}{\right\oldrfloor}
\renewcommand{\lceil}{\left\oldlceil}
\renewcommand{\rceil}{\right\oldrceil}

{\end{enumerate}}

\newcommand{\cA}{\mathcal{A}}
\newcommand{\cC}{\mathcal{C}}
\newcommand{\cE}{\mathcal{E}}
\newcommand{\cF}{\mathcal{F}}
\newcommand{\cG}{\mathcal{G}}
\newcommand{\cH}{\mathcal{H}}
\newcommand{\cK}{\mathcal{K}}
\newcommand{\cP}{\mathcal{P}}
\newcommand{\eps}{\varepsilon}
\renewcommand{\epsilon}{\varepsilon}
\newcommand{\Nat}{\mathbb{N}}
\newcommand{\Rat}{\mathbb{Q}}
\newcommand{\Real}{\mathbb{R}}

\def\Cfive{{\rm C5}}
\def\RL{{\rm RL}}
\def\Csev{{\rm C7}}

\DeclareMathOperator{\im}{im}

\DeclareMathOperator{\Hom}{Hom}

\newcommand{\T}{{\mathrm T}}
\newcommand{\unlab}[2]{\left\llbracket #1\right\rrbracket_{#2}}

\newcommand{\cond}{{ \; \big\vert \; }}

\newcommand{\linktocalculations}{\url{http://honza.ucw.cz/proj/EdgesInCycles/}}

\begin{document}

\title{Minimum number of edges that occur in odd cycles
\thanks{
This work has received funding from the European Research Council (ERC)
under the European Union's Horizon 2020 research and innovation programme
(grant agreement No 648509). This publication reflects only its authors'
view; the European Research Council Executive Agency is not responsible for
any use that may be made of the information it contains.
}
}

\author{
Andrzej Grzesik\thanks{Faculty of Mathematics and Computer Science, Jagiellonian University, ul. St. \L ojasiewicza 6, 30-348 Krak\'ow, Poland. Email: {\tt Andrzej.Grzesik@uj.edu.pl}. The author was partially supported by the National Science Centre grant 2013/08/T/ST1/00108.}
\and
Ping Hu\thanks{Current address: School of Mathematics, Sun Yat-sen University, Guangzhou, 510275, China. This work has been carried out while at: Department of Computer Science, University of Warwick, Coventry, CV4 7AL, United Kingdom.
Email: {\tt huping9@mail.sysu.edu.cn}.}
\and
Jan Volec\thanks{Department of Mathematics and Statistics, McGill University, Burnside Hall, 805 Sherbrooke West, Montreal H3A 2K6, Canada.
Former affiliation: Department of Mathematics, ETH, 8092 Z\"urich, Switzerland. Email: {\tt jan@ucw.cz}.
The author was partially supported by the SNSF grant 200021-149111 and CRM-ISM fellowship.}
}

\date{}

\maketitle

\begin{abstract}
If a graph $G$ has $n\ge 4k$ vertices and more than $n^2/4$ edges, then it
contains a copy of~$C_{2k+1}$. In 1992, Erd\H{o}s, Faudree and Rousseau showed
even more, that the number of edges that occur in a triangle of such a $G$ is
at least $2\lfloor n/2 \rfloor + 1$, and this bound is tight. They also showed
that the minimum number of edges in $G$ that occur in a copy of $C_{2k+1}$ for
$k \ge 2$ suddenly starts being quadratic in $n$, and conjectured that for any
$k\ge2$, the correct lower bound should be $2n^2/9 - O(n)$. Very recently,
F\"uredi and Maleki constructed a counterexample for $k=2$ and proved an
asymptotically matching lower bound, namely that for any $\eps > 0$
graphs with $(1+\eps)n^2/4$ edges contain at least $(2+\sqrt{2})n^2/16 \sim 0.2134n^2 $
edges that occur in $C_5$.

In this paper, we use a different approach to tackle this problem and prove
the following stronger result: Every $n$-vertex graph with at least $\lfloor
n^2/4 \rfloor + 1$ edges has at least $(2+\sqrt{2})n^2/16-O(n^{15/8})$ edges that
occur in $C_5$. Next, for all $k\ge 3$ and $n$ sufficiently large, we
determine the exact minimum number of edges that occur in $C_{2k+1}$
for $n$-vertex graphs with more than $n^2/4$ edges, and show it is indeed equal to 
$\lfloor \frac{n^2}4 \rfloor + 1 - \lfloor \frac{n+4}6 \rfloor \lfloor \frac{n+1}6
\rfloor = 2n^2/9 - O(n)$.
For both of these results, we give a structural description of all the large
extremal configurations as well as obtain the corresponding stability results,
which answers a conjecture of F\"uredi and Maleki. 

The main ingredient of our results is a novel approach that combines the
semidefinite method from flag algebras together with ideas from finite
forcibility of graph limits, which may be of independent interest. This
approach allowed us to keep track of the additional extra edge needed to
guarantee even the existence of a single copy of $C_{2k+1}$, which a standard
flag algebra approach would not be able to handle. Also, we establish the first
application of the semidefinite method in a setting, where the set of tight
examples has exponential size and arises from two very different
constructions.
\end{abstract}

\section{Introduction}
A classical result in graph theory is Mantel's Theorem~\cite{mantel:1907}, which states
that every triangle-free graph on $n$ vertices has at most $\lfloor n^2/4
\rfloor$ edges, and this result is tight. In other words, a graph with $n$
vertices and $\lfloor n^2/4 \rfloor+1$ edges must contain a triangle.
But can we guarantee something stronger than just one triangle?
In 1941, Rademacher proved that such graphs contain at least $\lfloor n/2\rfloor$ triangles,
and in 1992, Erd\H os, Faudree and Rousseau~\cite{ERDOS199223} showed that such graphs have 
at least $2\lfloor n/2 \rfloor + 1$ edges that occur in a triangle.
Both results are tight simply by adding one edge to the complete balanced bipartite graph.

Erd\H os~\cite{Erdos199781} also considered analogous questions for longer odd
cycles in graphs with $n$ vertices and $\lfloor n^2/4 \rfloor+1$ edges, where
 adding an extra edge into the complete balanced bipartite graph is not optimal.
He showed that every such graph contains at least $2n^2/9$ edges that occur in
some odd cycle. This number is best possible, and it can be achieved by the following construction.

\begin{construction}\label{cstn:cliquebip}
Let $G_1$ be an $n$-vertex graph with the following two $2$-connected blocks that overlap on exactly one vertex:
\begin{enumerate}
\item a complete graph on $\lfloor \frac{2n+4}{3}\rfloor$ vertices, and
\item a complete balanced bipartite graph on $\lfloor \frac{n+1}{3} \rfloor$ vertices.
\end{enumerate}
\end{construction}

\begin{figure}[ht]
\begin{center}
\begin{tikzpicture}[scale=1.25, very thick]
	\coordinate (AB) at (0,-2);
	\coordinate (AT) at (0,0);
	\coordinate (BB) at (2,-2);
	\coordinate (BT) at (2,0);

	\draw[fill=gray,gray] (AT) -- (BT) -- (BB) -- (AB);

	\draw[fill=gray] (+3.9,-1) ellipse (1.5 and 1.9); 
	\draw[fill=white] (0,-1) ellipse (.4 and 1); 
	\draw[fill=white] (2,-1) ellipse (.4 and 1); 
        \draw (+2.4,-1) node[inner sep=3pt, outer sep=0pt, circle, fill] {};
	
\end{tikzpicture}
\end{center}

\caption{The graph $G_1$ from Construction~\ref{cstn:cliquebip}. Gray areas represent all the possible edges.}
\label{fig:cliquebip}
\end{figure}
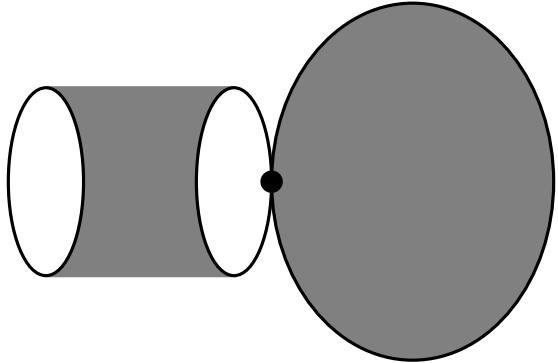

Erd\H os, Faudree and Rousseau~\cite{ERDOS199223} conjectured that 
Construction~\ref{cstn:cliquebip} provides an extremal example
also if we minimize the number of edges that occur only in copies of $C_{2k+1}$ for a fixed $k \ge 2$.
Again, we minimize over all $n$-vertex graphs with $\lfloor n^2/4\rfloor + 1$ edges.
The case of $C_5$ is Problem 11 in Erd\H os' paper \cite{Erdos199781} with interesting problems. 

\begin{conj}[Erd\H os-Faudree-Rousseau~\cite{ERDOS199223}]\label{conj:erdos}
Fix an integer $k\ge 2$. Every graph with $n$ vertices and $\lfloor \frac{n^2}4
\rfloor+1$ edges contains at least $\frac29 n^2 - O(n)$ edges that occur in
$C_{2k+1}$. 
\end{conj}


Very recently, F\"uredi and Maleki~\cite{bib:FurMal} constructed the following $n$-vertex graph
with $\lfloor n^2/4 \rfloor+1$ edges, out of which only $\frac{2+\sqrt{2}}{16}\cdot n^2 + O(n)
\approx 0.2134 n^2$  occur in $C_5$,
which disproves Conjecture~\ref{conj:erdos} for $k=2$.

\begin{construction}\label{cstn:c5}
Let $G_2$ be an $n$-vertex graph whose vertex-set is divided into
four parts $A,B,C$ and $D$ of sizes $\frac{2-\sqrt{2}}{4}\cdot n, \frac{n}{4}, \frac{n}{4}$ and $\frac{\sqrt{2}}{4}\cdot n$, respectively.
The edge-set of $G_2$ consists of all the edges between the parts $A$ and $B$, $B$ and $C$, $C$ and $D$, and all the edges inside the part $D$.
In other words, $G_2$ is a non-balanced blowup of a path on four vertices, where one of the endpoints of the path has a loop.
\end{construction}

\begin{figure}[ht]
\begin{center}
\begin{tikzpicture}[scale=1.5, very thick]
	\coordinate (AB) at (0,-1.5);
	\coordinate (AT) at (0,-0.5);
	\coordinate (BB) at (2,-2);
	\coordinate (BT) at (2,0);
	\coordinate (CB) at (4,-2);
	\coordinate (CT) at (4,0);
	\coordinate (DB) at (6,-2.3);
	\coordinate (DT) at (6,0.3);
	\draw[fill=gray,gray] (AT) -- (BT) -- (CT) -- (DT) -- (DB) -- (CB) -- (BB) -- (AB) -- (AT);
	\draw[fill=white] (0,-1) ellipse (.3 and 0.5); 
	\draw[fill=white] (2,-1) ellipse (.4 and 1); 
	\draw[fill=white] (4,-1) ellipse (.4 and 1); 
	\draw[fill=gray] (6,-1) ellipse (.6 and 1.3); 
        \node [below] at (0,-2.4) {$A$};
         \node [below] at (2,-2.4) {$B$};
          \node [below] at (4,-2.4) {$C$};
           \node [below] at (6,-2.4) {$D$};
	
\end{tikzpicture}
\end{center}

\caption{The graph $G_2$ from Construction~\ref{cstn:c5}.
Gray areas represent all the possible edges.
The respective sizes are $\frac{2-\sqrt{2}}{4}\cdot n$, $\frac{n}{4}$, $\frac{n}{4}$, and $\frac{\sqrt{2}}{4}\cdot n$.}
\label{fig:counterexample}
\end{figure}
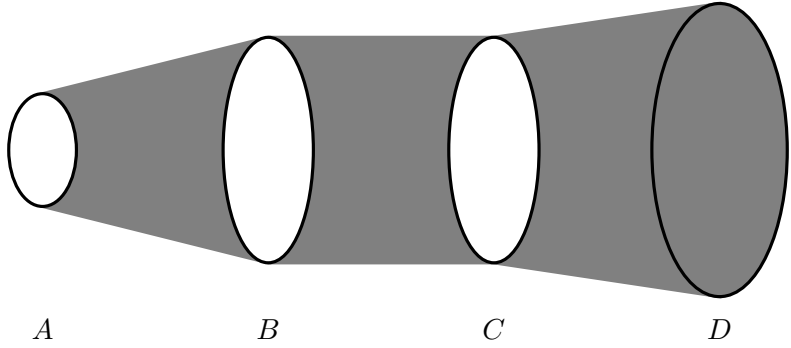

F\"uredi and Maleki~\cite{bib:FurMal} developed a new version
of Zykov's symmetrization method, and obtained the following asymptotic
solution to this problem for all odd cycles of length at least five.
\begin{theorem}[\cite{bib:FurMal}]\label{thm:FurMal}
For every $\eps > 0$, there exists $n_0\in \Nat$ such that
if a graph $G$ on $n>n_0$ vertices has $\left(\frac14+\eps\right)n^2$ edges,
then $G$ contains at least $\frac{2+\sqrt{2}}{16}\cdot n^2$ edges that occur in $C_5$.
Moreover, for any fixed $k\ge3$, $G$ contains at least $\frac29 \cdot n^2$ edges that occur in $C_{2k+1}$.
\end{theorem}

Our first two results answer a conjecture of F\"uredi and Maleki
that the assumption on the number of edges of $G$ can be
lowered to $\lfloor n^2/4 \rfloor + 1$, which is indeed best possible.
\begin{theorem}\label{thm:c5}
If an $n$-vertex graph has $\lfloor \frac{n^2}4\rfloor + 1$ edges, then it contains
at least $\frac{2+\sqrt{2}}{16} \cdot n^2-O\left(n^{15/8}\right)$ edges that occur in $C_5$.
\end{theorem}

\begin{theorem}\label{thm:c7+}
For every integer $k \ge 3$, if an $n$-vertex graph has $\lfloor \frac{n^2}4\rfloor
+1$ edges, then it contains at least $\frac29 \cdot n^2 - O(n)$ edges that occur in
$C_{2k+1}$.
\end{theorem}

In the case of odd cycles of length at least $7$ and $n$ sufficiently large,
we determine the exact value of the number of edges that occur in $C_{2k+1}$,
which indeed matches the value given by Construction~\ref{cstn:cliquebip},
which answers another conjecture of F\"uredi and Maleki~\cite[Conjecture 8]{bib:FurMalTria}.
\begin{theorem}\label{thm:c7++}
There exists $n_0 \in \Nat$ such that the following is true for any $n$-vertex
graph $G$ with $n \ge n_0$. If $G$ has $\lfloor \frac{n^2}4 \rfloor + 1$ edges, then
it contains at least $\lfloor \frac{n^2}4 \rfloor + 1 - \lfloor \frac{n+4}6 \rfloor \lfloor \frac{n+1}6
\rfloor$ edges that occur in
$C_{2k+1}$.  \end{theorem}

The main tool in our proofs is the semidefinite method from flag algebras,
which we apply in a specific $2$-edge-colored setting.  This approach has the
unfortunate by-product, that we lose track of the additional edge that is
needed to guarantee even the existence of a single copy of $C_{2k+1}$.  In order
to overcome this, we apply a trick inspired by techniques used in the area of
so-called finitely forcible graph limits. This allows us to obtain a tight
bound from flag algebras conditioned by having a positive triangle density, and
then handle the (almost) triangle-free case using a standard stability
argument.  A closely related difficulty of our approach arises from the fact that 
instead of most applications of the semidefinite method, where there is only one tight example, 
the flag algebra formulation of this problem has a significantly larger set
of tight examples. Nevertheless, we were still able to obtain a tight result in
this setting. 

We guided our method to establish a slightly stronger flag algebra claims which yield also the corresponding stability results:
\begin{thm}\label{thm:c5uniq}
For every $\eps > 0$ there exist $\delta > 0$ and $n_0\in\Nat$
such that the following is true for any $n > n_0$.
If $G$ is an $n$-vertex graph with $\left(\frac14 \pm \delta\right)n^2$ edges out of which
$\left(\frac{2+\sqrt{2}}{16} \pm \delta\right)n^2$ occur in $C_5$,
then the edge-set of $G$ can be modified on at most $\eps n^2$
pairs so that the resulting graph is isomorphic to Construction~\ref{cstn:c5}.
\end{thm}

\begin{thm}\label{thm:c7+uniq}
Fix an integer $k\ge3$. For every $\eps > 0$ there exist $\delta > 0$ and $n_0 \in \Nat$
such that the following is true for any $n > n_0$.
If $G$ is an $n$-vertex graph with $\left(\frac14 \pm \delta\right)n^2$ edges out of which $\left(\frac29 \pm \delta\right)n^2$ occur in
$C_{2k+1}$, then the edge-set of $G$ can be modified on at most $\eps n^2$
pairs so that the resulting graph is isomorphic to Construction~\ref{cstn:cliquebip}.
\end{thm}

Using the above stability results, we fully describe all the sufficiently large
graphs that contain the minimum value of edges that occur in odd cycles of
length at least $5$. The description of the tight graphs in the case of
pentagons is given by Theorem~\ref{thm:c5exact}. For all the longer odd cycles,
the description is provided by Theorem~\ref{thm:c7+exact}, which in turn proves~both Theorems~\ref{thm:c7+} and~\ref{thm:c7++}.


This paper is organized as follows. In Section~\ref{sec:prelim}, we describe
the notation and introduce parts from the flag algebra framework we are going to use.
In Section~\ref{sec:c5}, we present
our proof of Theorem~\ref{thm:c5}, and in Section~\ref{sec:c7+}, we adapt the
approach to cope with odd cycles of length at least $7$. 
Section~\ref{sec:stability} is devoted to the corresponding stability results
of the Constructions~\ref{cstn:cliquebip} and~\ref{cstn:c5}. Finally, in Sections~\ref{sec:c5exact} and \ref{sec:c7+exact} we provide the exact description of the tight extremal graphs.
Section~\ref{sec:remarks} concludes the paper with remarks and related open
problems.


\section{Notation and preliminaries}
\label{sec:prelim}

We start with the definition of the \emph{induced density} of a $k$-vertex
(small) graph $F$ in an $n$-vertex (large) graph $G$, which we denote by $p(F,G)$.
If $n \ge k$, then  $p(F,G)$ is the probability that a randomly chosen
$k$-vertex induced subgraph of $G$ is isomorphic to $F$.
In the case when $k > n$, the value of $p(F,G)$ is simply equal to zero.

In order to distinguish the edges that occur in some copy of $C_5$ (or
more generally $C_{2k+1}$ for some fixed $k\ge2$) in graphs $G$ in question, we will
work with edge-colorings of $G$ where the edges are colored using two
colors -- \emph{red} and \emph{blue}.
With a slight abuse of notation, we will use $G$ both to refer to the underlying
graph and to the edge-colored graph, whenever it will be clear from the context
which variant we intend to use. A graph $G$ with edges colored by red and blue will be 
called a \emph{red/blue-colored graph}.
Through the whole paper, we will use the convention that
none of the blue edges of $G$ can occur in a copy of $C_{2k+1}$ for a given $k \ge 2$. 
Let us emphasize that we do not put any restriction on the red edges of $G$,
so in particular any graph $G$ can be completely colored with red.

The definition of the induced density $p(F,G)$ naturally generalizes to the
edge-colored setting.  For convenience, we extend the definition of $p(F,G)$
also to graphs $F$ where we allow the edges to be colored with three colors --
red, blue or black (the edges of $G$ will always be colored only with red and
blue). The interpretation of an edge of $F$ being black will be that we
do not care whether $G$ contains a copy of $F$ where the edge is colored red
or blue. Therefore, for a $k$-vertex red/blue/black-colored graph $F$
and an $n$-vertex red/blue-colored graph $G$, the value of $p(F,G)$
is the probability that a random $k$-vertex subgraph of $G$ is isomorphic to
one of the graphs that can be obtained from $F$ by recoloring each of its
black edges to either red or blue.

We depict red/blue/black-colored graphs in the following way. We draw black edges
using solid lines, for blue edges we used dashed lines, and finally red edges will
be depicted using dotted lines; see Figure~\ref{fig:edges}.

\begin{figure}
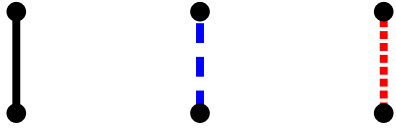

\begin{center}
\includegraphics[scale=1.5,page=19]{EiC-fig}
\hskip 2cm 
\includegraphics[scale=1.5,page=3]{EiC-fig}
\hskip 2cm 
\includegraphics[scale=1.5,page=4]{EiC-fig}
\end{center}
\caption{Our convention used for depicting black, blue and red edges -- black edges are drawn with
solid lines, blue edges with dashed lines and red edges with dotted lines.}
\label{fig:edges}
\end{figure}

\subsection{\emph{F}-free graphs and sequences of almost \emph{F}-free graphs}

We will also use the following notion of \emph{$F$-free graphs} and \emph{sequences of almost $F$-free graphs}.
For a $k$-vertex graph $F$ and a graph $G$, we say that a graph $G$ is \emph{$F$-free},
if $G$ does not contain $F$ as a subgraph.
For a sequence of graphs $(G_i)_{i\in\Nat}$, where the $i$-th graph $G_i$ has $n_i$ vertices,
we say that $(G_i)_{i\in\Nat}$ is almost $F$-free, if $G_i$ contains only $o\left(n_i^k\right)$ copies of $F$.
This notion naturally generalizes to the red/blue-colored setting.

If $\cF$ is a finite collection of graphs, we say that $G$ is $\cF$-free
and $(G_i)_{i\in\Nat}$ is almost $\cF$-free, if $G$ is $F$-free for every $F \in \cF$
and $(G_i)_{i\in\Nat}$ is almost $F$-free for every $F \in \cF$, respectively.
We also extend the notion of being $F$-free to red/blue/black-colored
graphs $F$, where being $F$-free corresponds to being $\cF(F)$-free, where $\cF(F)$
denotes the family of red/blue-colored graphs consisting of all the possible recolorings of the black
edges in $F$ by red or blue. Analogously, we extend the notion of being almost $F$-free,
and the notions of $\cF$-free and almost $\cF$-free for finite families $\cF$ consisting of red/blue/black-colored graphs.

Now let us recall a classical generalization of the theorem of K\H{o}vari, S\'os and Tur\'an
to $r$-uniform hypergraphs (or just \emph{$r$-graphs} for short) which is due to Erd\H{o}s~\cite{Erdos:1964}.
\begin{thm}
\label{thm:erdosKST}
If $H$ is an $r$-graph on $n$ vertices with no copy of the complete $r$-partite $r$-graph that has all the parts of size $\ell$,
then the number of $r$-edges in $H$ is at most $O\left(n^{r-1/\ell^{(r-1)}}\right)$.
\end{thm}

A standard averaging argument together with Theorem~\ref{thm:erdosKST} yields the following result on supersaturation in dense graphs,
which will be one of the ingredients we will use in the proofs of Theorems~\ref{thm:c5} and~\ref{thm:c7+}.
\begin{cor}
\label{cor:supersat}
Fix $F$ an $h$-vertex red/blue-colored graph and a positive integer $b$. 
If $G$ is an $n$-vertex red/blue-colored graph that does not contain the $b$-blowup of $F$ as a subgraph,
then the number of copies of $F$ in $G$ is $O\left(n^{h-1/b^{(h-1)}}\right)$.
\end{cor}

\subsection{Flag Algebras}

In this subsection, we describe parts of the flag algebra framework of
Razborov~\cite{Razborov:2007} that will be relevant for our exposition.  Flag
algebras play a crucial role in our proofs of Theorems~\ref{thm:c5}
and~\ref{thm:c7+}.  We follow the notation from~\cite{Razborov:2007} with a few
minor alternations specific to sequences of almost $\cF$-free graphs.
Flag algebras have been very successful in tackling various problems.
To mention some of them:
Caccetta-H\"aggkvist conjecture~\cite{Grzesik:2017,HladkyKN:2009,RazborovCH:2011},
various Tur\'an-type problems in graphs~\cite{DasHMNS:2012, Grzesik:2011,Hatami:2011,Hirst:2014,Nikiforov:2011, PikhurkoR:2012,PikhurkoV:2013,Razborov:2008,Reiher:2012,Sperfeld:2011},
hypergraphs~\cite{BaberT:2011,Falgas:2012,Falgas:2011,GlebovKV:2013,Pikhurko:2011}
and hypercubes~\cite{Baber:2012,BaloghHLL:2014},
extremal problems in a colored environment~\cite{BaberT:2013,CummingsKPSTY:2012,HatamiJKNR:2012,KralLSWY:2012}
and also to problems in geometry~\cite{Kral:2011} or extremal theory of permutations~\cite{BaloghHLPUV:2014}.
For more details on these applications, see a survey of Razborov~\cite{Razborov13}.

The central object of interest in flag algebras are so-called \emph{convergent
sequences} of finite discrete objects, for example finite graphs. In this
paper, we apply the framework to sequences of red/blue-colored almost
$\cF$-free graphs, for two certain choices of $\cF$ (the two families will be
explicitly specified in Sections~\ref{sec:c5} and~\ref{sec:c7+}, respectively).


In the following, we describe a~\emph{flag algebra} $\cA^{\sigma}$, where
$\sigma$ is a fixed vertex-labelled red/blue-colored graph,
on all the red/blue-colored graphs  with a fixed copy of a labelled graph
$\sigma$.  The graph $\sigma$ is usually called a~{\em type}.
Note that we will use simply $\cA$ to refer to the algebra $\cA^\emptyset$,
where $\emptyset$ is the empty type.

Fix a type $\sigma$.
Let $\cH^\sigma$ be the set of all finite red/blue-colored graphs with a fixed
{\em embedding} of $\sigma$, i.e., an injective mapping $\theta$ from
$V(\sigma)$ to $V(H)$ such that $\theta$ is an isomorphism between $\sigma$ and
$H[\im(\theta)]$. The elements of $\cH^{\sigma}$ are called {\em
$\sigma$-flags}, and the subgraph induced by $\im(\theta)$ is called the {\em
root} of a $\sigma$-flag.
For every $\ell\in\Nat$, we let $\cH^\sigma_\ell$ to be the subset of $\cH^\sigma$ containing all of its $\ell$-vertex graphs.
Let $\Real\cH^\sigma$ be the set of all formal linear combinations of the
$\sigma$-flags with real coefficients, and $\cK^\sigma$ the linear subspace of
$\Real\cH^\sigma$ generated by all the combinations of the form
\[H-\sum_{H'\in\cH^\sigma_{v(H)+1}}p(H,H')\cdot H'.\]

The algebra $\cA^\sigma$ is defined as $\Real\cH^\sigma$ factored by $\cK^\sigma$, and 
the element corresponding to $\cK^\sigma$ in $\cA^\sigma$ is the zero element of $\cA^\sigma$.
$\cA^\sigma$ comes with a natural definition of the addition; the notion of multiplication is slightly more involved.
Firstly, we describe a product of two $\sigma$-flags $H_1 \in \cH^\sigma_k$ and $H_2 \in \cH^\sigma_\ell$.
For a $\sigma$-flag $H \in \cH^\sigma_{k+\ell-v(\sigma)}$ with $\theta$ being the
fixed embedding of $\sigma$, we define $p(H_1, H_2; H)$ to be the probability
that a randomly chosen subset of $V(H)\setminus \theta(V(\sigma))$ of size
$k-v(\sigma)$ and its complement in $V(H)\setminus \theta(V(\sigma))$ of
size $\ell-v(\sigma)$ extend $\theta(V(\sigma))$ in $H$ to $\sigma$-flags
isomorphic to $H_1$ and $H_2$, respectively.
We set
\[H_1 \times H_2 := \sum_{H\in\cH^\sigma_{k+\ell-v(\sigma)}}p(H_1,H_2;H) \cdot H,\]
and extend the notion linearly to the elements of $\cA^\sigma$.
Note that the unique $\sigma$-flag of order $v(\sigma)$ is, modulo $\cK^\sigma$, the neutral element of the product in $\cA^\sigma$.

Fix a finite family $\cF$ of red/blue-colored graphs and an $\cF$-free type $\sigma$.
The presented exposition of the flag algebra $\cA^\sigma$ on red/blue-colored
graphs naturally adapts to the setting of $\cF$-free red/blue-colored graphs,
simply by replacing the set $\cH^\sigma$ with the set of all red/blue-colored
$\cF$-free $\sigma$-flags.



Now consider an infinite sequence $(G_i)_{i\in\Nat}$ of red/blue-colored almost $\cF$-free
graphs with increasing orders.  We call the sequence {\em convergent} if the
probabilities $p(H,G_i)$ converge for every $H\in\cH$.  By compactness, every
sequence $(G_i)_{i\in\Nat}$ has a convergent subsequence.
For the rest of this section, we will assume that $(G_i)_{i\in\Nat}$ is convergent.
For $H\in\cH$, we set $\phi(H) = \lim_{i\to\infty} p(H,G_i)$, and linearly extend $\phi$ to the elements of $\cA$.
We refer to the mapping $\phi$ as to the {\em limit} of the sequence.
For every red/blue-colored graph $H$, it holds that $\phi(H)\geq 0$.
Moreover, the sequence $(G_i)_{i\in\Nat}$ is almost $\cF$-free, hence $\phi(F) = 0$ for all $F \in \cF$.
It follows that the restriction of $\phi$ to $\cF$-free red/blue-colored graphs is in fact
an algebra homomorphism from $\cA_\cF$ to $\Real$.
We let $\Hom^+(\cA_\cF, \Real)$ to be the set of all homomorphisms $\psi$ from
$\cA_\cF$ to $\Real$ such that $\psi(H)\ge0$ for every $\cF$-free $H\in\cH$.

For an $\cF$-free type $\sigma$ and its embedding $\theta$ in $G_i$,
we define $G_i^\theta$ to be the red/blue-colored graph rooted on $\theta$.
For every $i\in\Nat$ and $H^\sigma \in \cH^\sigma$, let
$p^\theta_i(H^\sigma)=p(H^\sigma,G_i^\sigma)$.
Picking $\theta$ at random gives rise to a probability distribution ${\bf P}_{\bf i}^\sigma$ on mappings
from $\cA^{\sigma}$ to $\Real$. 
It holds that the sequence of $\left({\bf P}_{\bf i}^\sigma\right)_{i\in\Nat}$ weakly converges
to a Borel probability measure on $\Hom^+(\cA^\sigma,\Real)$, see~\cite[Theorems 3.12 and 3.13]{Razborov:2007}.
In fact, for any $\sigma$ such that $\phi(\sigma) > 0$, the homomorphism $\phi$
fully determines the limit probability distribution~\cite[Theorem 3.5]{Razborov:2007}.
We denote the limit of $\left({\bf P}_{\bf i}^\sigma\right)_{i\in\Nat}$ by ${\bf P}_\phi^\sigma$.
Furthermore, since $(G_i)_{i\in\Nat}$ is almost $\cF$-free, any mapping
$\phi^\sigma$ drawn from the support of the distribution ${\bf P}_\phi^\sigma$
is in fact an algebra homomorphism from $\cA^{\sigma}_\cF$ to $\Real$ such that
$\phi^\sigma(H^\sigma) \ge 0$ for any $\sigma$-flag $H^\sigma$. 


The last notion we introduce is the \emph{averaging operator} (also called the \emph{downward operator})
$\unlab\cdot{\sigma}: \cA^{\sigma}_\cF \to \cA_\cF$.
It is a linear operator defined on the~$\sigma$-flags $H^\sigma$ by
\[\unlab{H^\sigma}{\sigma} := p_H^\sigma \cdot H,\] where
$H$ is the unlabelled red/blue-colored graph corresponding to $H^\sigma$, 
and $p_H^\sigma$ is the probability that a random injection from $V(\sigma)$ to $V(H)$ yields
a $\sigma$-flag isomorphic to $H^\sigma$. A key relation
is
\begin{equation}
\label{eq:flag:averaging}
\forall A^\sigma\in\cA^\sigma_\cF,\quad \phi\left(\unlab{A^\sigma}{\sigma}\right)=
\phi\left(\unlab\sigma\sigma\right) \cdot \int_{\phi^\sigma} \phi^\sigma(A^\sigma) \; d{\bf P}_\phi^\sigma\,
.
\end{equation}
%
If $\phi^\sigma(A^\sigma)\ge 0$ with probability one for some $A^\sigma \in \cA_{\cF}^\sigma$, then (\ref{eq:flag:averaging}) yields
that $\phi\left(\unlab{A^\sigma}{\sigma}\right)\ge 0$.
In~particular,
$
\phi\left(\unlab{A^\sigma \times A^\sigma }{\sigma}\right)\ge 0
$
for every $\phi \in \Hom^+(\cA_\cF,\Real)$ and every $A^\sigma\in\cA^\sigma_\cF$.


\section{Edges that occur in pentagons --- proof of Theorem~\ref{thm:c5}}
\label{sec:c5}

We start the proof by formulating the statement of Theorem~\ref{thm:c5}
into the language of red/blue-colored graphs. This statement is
convenient for the flag algebra framework, which we intend to apply.

\begin{thm}\label{thm:c5c}
If $G$ is a red/blue-colored graph on $n$ vertices with $\lfloor \frac{n^2}4 \rfloor + 1$ edges
and no blue edge occur in $C_5$, then $G$ contains at least $\frac{2+\sqrt{2}}{16}\cdot n^2 - O(n^{15/8})$ red edges.
\end{thm}

It is straightforward to check that the statements of Theorem~\ref{thm:c5} and Theorem~\ref{thm:c5c} are equivalent.
In the rest of the section, we give a proof of Theorem~\ref{thm:c5c}. We split the proof into the following two cases: either
$G$ contains many triangles and then we apply flag algebras, or,
$G$ contains only a small number of triangles in which case we use stability to
show that $G$ is close to the complete bipartite graph. 
Since $G$ has more than $n^2/4$ edges, it follows
that in the second case $G$ must have many red edges (in fact, more than Theorem~\ref{thm:c5c} asks for).

\subsection{Flag algebra setting}
We start with describing the precise setting of flag algebras we are going to
use. Clearly, every $G$ from Theorem~\ref{thm:c5c} is $B_5$-free, where $B_5$
is the $5$-cycle with one blue and four black edges (recall that an edge is
black if it is either red or blue).  But we can say more. Suppose $F$ is a
red/blue-colored graph such that the $b$-blowup of $F$, for some positive
integer $b$, contains $C_5$ with at least one blue edge. Then, by
Corollary~\ref{cor:supersat}, $G$ can contain only $O\left(n^k\right)$ copies
of such a graph $F$, where $k$ is a rational strictly smaller than $v(F)$ and
depends only on $F$ and $b$.
For example, since the $2$-blowup of the graph $B_3$ depicted in
Figure~\ref{fig:FC5graphs} contains $C_5$ with at least one blue edge, $G$
contains only $O\left(n^{3-1/4}\right)$ copies of $B_3$.  In other words,
all but $O\left(n^{11/4}\right)$ triangles in $G$ have only red edges.
Analogously, the $2$-blowup of the graph $B_3^+$, which is also depicted in
Figure~\ref{fig:FC5graphs}, contains $C_5$ with two blue edges.
Therefore, $G$ may contain only $O\left(n^{31/8}\right)$ copies of $B_3^+$.

\begin{figure}
\begin{center}
\includegraphics[scale=1,page=14]{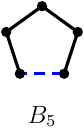}
\hskip 1.5cm 
\includegraphics[scale=1,page=11]{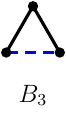}
\hskip 1.5cm 
\includegraphics[scale=1,page=12]{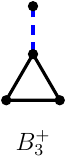}
\hskip 1.5cm 
\includegraphics[scale=1,page=16]{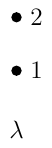}
\hskip 1.5cm 
\includegraphics[scale=1,page=17]{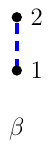}
\hskip 1.5cm 
\includegraphics[scale=1,page=18]{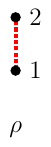}
\end{center}
\caption{The graphs $B_5$, $B_3$ and $B_3^+$ used in the construction of the
$\cF_{\Cfive}$-free flag algebra $\cA_{\Cfive}$, and the types $\lambda$,
$\beta$ and $\rho$.}
\label{fig:FC5graphs}
\end{figure}


Let $\cF_\Cfive := \left\{B_3, B_3^+, B_5 \right\}$.
For brevity, we will write $\cA_\Cfive$ and $\cA^\sigma_\Cfive$ instead of $\cA_{\cF_\Cfive}$ and $\cA^\sigma_{\cF_\Cfive}$.
Next, let $\lambda, \beta$ and $\rho$ be the three red/blue-colored flag algebra
types of size two with labels $1$ and~$2$, where $\lambda$ denotes the non-edge type,
$\beta$ the blue-edge type, and $\rho$ the red-edge type; see Figure~\ref{fig:FC5graphs}.
We define $\cH^\lambda_4$ to be the set of all the non-isomorphic $4$-vertex $\lambda$-flags in $\cF_\Cfive^\lambda$,
$\cH^\beta_4$ the set of the $4$-vertex $\beta$-flags in $\cF_\Cfive^\beta$,
and $\cH^\rho_4$ the set of the $4$-vertex $\rho$-flags in $\cF_\Cfive^\rho$.
It holds that $\left|\cH^\lambda_4\right|=76$, $\left|\cH^\beta_4\right|=33$, and $\left|\cH^\rho_4\right|=43$.
Let $v_\lambda$ be the $76$-dimensional vector in $\left(\cA^\lambda_\Cfive\right)^{76}$ such that
the $i$-th element of $v_\lambda$ is equal to the $i$-th element of $\cH^\lambda_4$. Analogously,
we let $v_\beta$ to be the $33$-dimensional vector, where the coordinates correspond
to the elements of $\cH^\beta_4$, and $v_\rho$ the $43$-dimensional vector, where the
coordinates correspond to the elements of $\cH^\rho_4$.
Finally, we define $\cH_6$ to be the set of all $6$-vertex red/blue-colored $\cF_\Cfive$-free graphs;
it can be checked that there are precisely $756$ such graphs.


Theorem~\ref{thm:c5c} is concerned with red/blue-colored graphs of density at least $1/2$,
which translates to the flag algebra framework as studying 
homomorphisms $\psi \in \Hom^+\left(\cA_{\Cfive},\Real\right)$ satisfying
$\psi\left(\figR \,+\, \figB\right) \ge 1/2$. As $\psi\left(\figN \,+\, \figR \,+\, \figB \right) = 1$, the last condition is equivalent 
to $\psi\left(\figR \,+\, \figB \,-\, \figN \right) \ge 0$.
%

\begin{prop}
\label{prop:c5flag}
There exist three positive-semidefinite matrices $L$, $B$ and $R$
with the entries from the field $\Rat\left[\sqrt 2\right]$,
and non-negative numbers $a\in\Rat\left[\sqrt 2\right]$, $b\in\Rat\left[\sqrt 2\right]$ and $c_H\in\Rat\left[\sqrt 2\right]$ for $H \in \cH_6$,
such that in the algebra $\cA_\Cfive$, the expression 
\[
  \unlab{v_\lambda^\T L v_\lambda}\lambda
+ \unlab{v_\beta^\T B v_\beta}\beta
+ \unlab{v_\rho^\T R v_\rho}\rho
+ \left(\figR \,+\, \figB \,-\, \figN  \right) \times \left(a\cdot\figEMP \,+\, b\cdot\figCOCh\right)
+ \sum\limits_{H \in \cH_6} c_H \cdot H
\]
is equal to
\[
\figTRG \,\times\,
\left(
8\cdot\figR - \left({2+\sqrt{2}}\right) \cdot \figV \,
\right)
\mbox{.}
\]
Moreover, if $H$ is a $6$-vertex red/blue-colored graph from $\cH_6$ that contains an induced copy of any $P_5 \in \cP_5$
or $C_4^X$, then $c_H > 0$.
\end{prop}

Finding the positive-semidefinite matrices $L$, $B$ and $R$, and the non-negative numbers $a$, $b$ and $c_H$ 
for $H \in \cH_6$ such that the claimed identity holds can be expressed as a semidefinite program.
We used an SDP solver called CSDP~\cite{Borchers:1999} together with a computer algebra software SAGE~\cite{sw:sage} to help us solving
the corresponding semidefinite program.
Since some of the numbers and the entries of the matrices are too large to be presented in a printed form,
we created a webpage  and uploaded all the corresponding data there.
The URL of the webpage is \mbox{\linktocalculations}.

We also prepared a short verification script in SAGE that checks the
correctness of the claimed identity; see also Appendix~\ref{apx:verify} for the
details about the formal verification.  The script, as well as a description
of all the data files, can be also found at the webpage.
Note that the matrices $L$, $B$ and $R$ are not stored directly. Instead, they are decomposed as
\[
L = M_\lambda^\T \cdot \widehat{L} \cdot M_\lambda,\quad
B = M_\beta^\T \cdot \widehat{B} \cdot M_\beta,\quad \hbox{\rm{and}} \quad
R = M_\rho^\T \cdot \widehat{R} \cdot M_\rho,
\]
where $\widehat{L}$, $\widehat{B}$ and $\widehat{R}$ are positive definite matrices of sizes $50 \times 50$, $19 \times 19$ and $26\times 26$,
respectively, and $M_\lambda$, $M_\beta$ and $M_\rho$ are specific matrices of sizes $76\times 50$, $33\times 19$ and $43\times 26$, respectively.
Another advantage of this is that verifying whether a matrix is positive definite is faster from the practical point of view;
see Appendix~\ref{apx:verify}.

\subsection{Case 1 --- Graphs with many triangles}
\label{sec:c5case1}

We first prove the theorem for graphs $G$ that satisfy the assumptions of Theorem~\ref{thm:c5c} and contain $\Omega\left(n^3\right)$
triangles. This will be the only case where we use flag algebra method, and the reason for that is the following.
In order to apply flag algebra method, we pass the asymptotic statement to the limit. As we already mentioned in the
introduction, an unfortunate consequence is that we completely lose control on having the additional edge that
is needed to contain even a single copy of $C_5$. However, in the situation that $G$ contains about $n^2/4$ edges
and only a small number of triangles, a stability argument yields that $G$ must be very close to the complete balanced bipartite
graph. Such a situation will be analyzed in Section~\ref{sec:c5case2}.
Therefore, the statement we prove with flag algebras states that for every
$G$ that satisfies the assumptions of Theorem~\ref{thm:c5c}, at least one of
the following is true:
\begin{enumerate}
\item $G$ has at least $\frac{2+\sqrt{2}}{16} \cdot n^2 - O(n^{15/8})$ red edges, or,
\item $G$ contains $o\left(n^3\right)$ triangles.
\end{enumerate}

Suppose, for a contradiction, that Theorem~\ref{thm:c5c} is false.
Then there exists a sequence of red/blue-colored graphs
$(G_i)_{i\in\Nat}$ of increasing orders $n_i$ such that for $i$ big enough every $G_i$ has
at~most~\hbox{$\frac{2+\sqrt{2}}{16}\cdot n_i^2 - \omega\left(n_i^{15/8}\right)$} red edges.
Without loss of generality, the sequence is convergent.
Furthermore, the sequence $(G_i)_{i\in\Nat}$ is almost $\cF_{\Cfive}$-free.
Therefore, the sequence converges to a limit $\phi_0$, which is an element of
the set $\Hom^+\left(\cA_\Cfive,\Real\right)$.  It is straightforward to check
that the edge-density of $\phi_0$ is equal to $1/2$.  The following lemma
states that such a limit $\phi_0$ must have triangle density equal to zero.

\begin{lem}\label{lem:c5flag}
Let $\delta > 0$ and $\phi \in \Hom^+\left(\cA_{\Cfive},\Real\right)$.
If $\phi\left(\figR \,+\, \figB\right) \ge \frac12$ and $\phi\left(\figTRG\right) \ge \delta$,
then $\phi\left(\figR\right) \ge \frac{2+\sqrt{2}}8$. 
Moreover,
if $G$ is an $n$-vertex red/blue-colored graph with $\lfloor \frac{n^2}4 \rfloor + 1$ edges,
at least $\delta \cdot n^3$ triangles and no blue edge occur in $C_5$, 
then $G$ contains at least $\frac{2+\sqrt{2}}{16}\cdot n^2 - O(n^{15/8})$ red edges.
\end{lem}

\begin{proof}
Proposition~\ref{prop:c5flag} yields that if
$\psi \in \Hom^+\left(\cA_{\Cfive},\Real\right)$ satisfies $\psi\left(\figR \,+\, \figB\right) \ge 1/2$,
then
\begin{equation}
\psi \left( \figTRG \,\times\, \left( 8\cdot\figR - {2-\sqrt{2}} \right) \right) \ge 0
\mbox{.}
\label{eq:c5flag}
\end{equation}
It immediately follows that if the first factor of the product on the left-hand
side, i.e., the triangle density in $\psi$, is at~least $\delta > 0$, then 
inequality (\ref{eq:c5flag}) yields that the second factor must be non-negative. In other words,
the density of red edges is at~least $\left(2+\sqrt{2}\right)/8$. 


The moreover part of the lemma follows from a standard $O\left(n^{-1}\right)$ error estimate in the semidefinite method
(for details, see, for example, \cite{Oleg}), and the $O\left(n^{-1/8}\right)$ estimate on the densities of $B_3$ and $B_3^+$
in $G$.
\end{proof}

Recall that $(G_i)_{i \in \Nat}$ is a sequence of $n_i$-vertex graphs with at
most~\hbox{$\frac{2+\sqrt{2}}{16} \cdot n_i^2 - \omega\left(n_i^{15/8}\right)$}
red edges.  Applying Lemma~\ref{lem:c5flag} to $(G_i)$ readily implies that 
$G_i$ must contain $o\left(n_i^3\right)$ triangles.



\subsection{Case 2 --- Graphs with small number of triangles}
\label{sec:c5case2}

It remains to verify Theorem~\ref{thm:c5c} for graphs $G$ that
contain less than $\delta n^3$ triangles for an arbitrary $\delta > 0$.
As we have already mentioned, in this case our plan is to use stability of triangle-free
graphs to show that $G$ must be close, in the so-called \emph{edit-distance}, to a complete
bipartite graph. Since the number of edges in $G$ is strictly more than
$n^2/4$, the graphs we are dealing with are essentially almost complete bipartite graphs plus an additional
edge in one of the parts. Therefore, we will be able to show that nearly all
the edges of $G$ occur in $C_5$. This is summarized in the following lemma,
which actually holds for any odd cycle of length at least five.

\begin{lem}
\label{lem:stab}
For every integer $k\ge 2$, there exists $\delta_k > 0$ such that the following is true.
If $G$ is an $n$-vertex graph with $\lfloor \frac{n^2}4\rfloor+1$ edges and
at most $\delta_k \cdot n^3$ triangles,
then all but $o\left(n^2\right)$ edges of $G$ occur in $C_{2k+1}$.
\end{lem}


Before proving Lemma~\ref{lem:stab}, let us recall two classical results in extremal graph theory.
The first one is the triangle removal lemma due to Ruzsa and Szemer\'edi~\cite{ruzszem:1978}.
\begin{thm}
\label{thm:trgremoval}
If an $n$-vertex graph has $o(n^3)$ triangles, then it
can be made triangle-free by removing at most $o(n^2)$ edges.
\end{thm}
Next, we recall a classical stability-type result for dense triangle-free
graphs; for its proof, see, e.g., \cite[Theorem VI.4.2]{bib:bollobook}.
\begin{thm}
\label{thm:trgstability}
If $G$ is an $n$-vertex triangle-free graph with $n^2/4 - o(n^2)$ edges,
then $G$ contains an induced bipartite subgraph with minimum degree $n/2 - o(n)$.
\end{thm}

We are now ready to prove the main lemma of this subsection.

\begin{proof}[Proof of Lemma~\ref{lem:stab}]

Suppose that $\delta_k$ is sufficiently small.
An application of Theorems~\ref{thm:trgremoval} and~\ref{thm:trgstability} to $G$
readily finds an induced bipartite subgraph on $n-o(n)$ vertices with minimum degree $\frac{n}2 - o(n)$.

Let $G_0$ be an induced bipartite subgraph of $G$ with maximum number of vertices that has the minimum degree at least $\frac{n}2-o(n)$.
Let $A$ and $B$ be the parts of $G_0$ and let $L := V(G) \setminus V(G_0)$.
Clearly, both $A$ and $B$ have sizes between $\frac{n}2 \pm o(n)$ and $|L| = o(n)$.

The following claim states that $G_0$ has many edges between any two large subsets of $A$ and $B$.
\begin{claim}
If $A' \subseteq A$ and $B' \subseteq B$ are two sets of vertices of size $n/2-o(n)$ each,
then the number of edges in $G_0$ between the sets $A'$ and $B'$ is $n^2/4-o(n^2)$.
\end{claim}
\begin{proof}
Let $e$ be the number of edges between $A'$ and $B'$.
On one hand, $G_0$ has at least $n^2/4 - o(n^2)$ edges.
On the other hand, $e(G_0) \le e + |A\setminus A'|\cdot|B| + |B\setminus B'|\cdot|A| \le e + o(n^2)$.
\end{proof}

Our second claim states that any vertex with neighbors both in $A$ and $B$ allows
us to find many edges that occur in $C_{2k+1}$.
\begin{claim}
If some vertex $v_L\in L$ is adjacent both to a vertex $v_A \in A$ and a vertex $v_B \in B$,
then all but $o(n^2)$ edges of $G$ occur in $C_{2k+1}$.
\end{claim}
\begin{proof}
Let $V_{2k-1}\subseteq B$ be the neighborhood of $v_A$ in $G_0$,
and $v_B,v_2,\dots,v_{2k-3}$ any path in $G_0$ that does not contain
the vertex $v_A$. We set $V_{2k-2}\subseteq A$ to be the neighborhood of $v_{2k-3}$ in $G_0$,
$A' := V_{2k-2} \setminus \{v_2,\dots,v_{2k-2},v_A\}$
and $B' := V_{2k-1} \setminus \{v_B,v_3,\dots,v_{2k-3}\}$.

It follows that  both $A'$ and $B'$ have size at least $\frac{n}2 - o(n)$.
Therefore, the number of edges of the form $\{v_{2k-2},v_{2k-1}\}$ between $A'$ and $B'$ is
at least ${n^2}/4 - o(n^2)$ Each such an edge encloses a $(2k+1)$-cycle in
$G$, which is of the form $v_L,v_B,v_2,\dots,v_{2k-1},v_A$.
\end{proof}

In order to finish the proof of the lemma, we simply need to find such a vertex $v_L$.
Firstly, observe that $|L|\ge 1$, as otherwise $G$ cannot have $\lfloor n^2/4\rfloor + 1$ edges.
Moreover, at least one $v_L \in L$ must have $\deg_G(v_L) \ge n/2$.
If $v_L$ would be adjacent only to $A$ or only to $B$, then the subgraph $G_0+v_L$ contradicts the choice of $G_0$.
\end{proof}

Now recall that if Theorem~\ref{thm:c5c} would be false, then by
Lemma~\ref{lem:c5flag} there exists a limit $\phi_0 \in
\Hom^+\left(\cA_\Cfive,\Real\right)$ such that $\phi_0\left(\figTRG\right) = 0$
and $\phi_0\left(\figR\right) \le \frac{2+\sqrt{2}}8$.  However,
Lemma~\ref{lem:stab} yields that $\phi_0\left(\figR\right) = \frac12$.  Therefore,
there is no such $\phi_0$, and the proof Theorem~\ref{thm:c5c} is now finished.


\section{Edges that occur in longer odd cycles --- proof of Theorem~\ref{thm:c7+}}
\label{sec:c7+}

We adapt the approach presented in the previous section and give a proof of an
asymptotic version of Theorem~\ref{thm:c7+}. The exact version will be obtained
in Section~\ref{sec:c7+exact}, where we find a description of all the sufficiently
large extremal constructions.
We start with stating the main result of this section using the language of red/blue-colored graphs.
\begin{thm}\label{thm:c7+c}
For every $\eps > 0$ and integer $k \ge 3$, there exists $n_0\in \Nat$ such that 
if $G$ is a red/blue-colored graph on $n>n_0$ vertices with $\lfloor \frac{n^2}4 \rfloor + 1$ edges
and no blue edge occur in $C_{2k+1}$, then $G$ contains at least $\left(\frac29 -\eps\right)n^2$ red edges.
\end{thm}

The rest of this section is devoted to the proof of Theorem~\ref{thm:c7+c}.
First, we define a class of graphs $\cF_\Csev$ such that, for any fixed integer $k\ge3$,
the following will be true. If a sequence $(G_i)_{i\in\Nat}$ of red/blue-colored graphs is such that no blue edge occurs in $C_{2k+1}$,
then $(G_i)$ is almost $\cF_\Csev$-free.
Analogously to the $C_5$ case, a $k$-blow of any graph $F \in \cF_\Cfive$ contains a copy of $C_{2k+1}$ with at least one blue edge.
Similarly, a $k$-blowup of either the graph $B_3^*$ or the graph $B_5^+$, which are both depicted in Figure~\ref{fig:FC7graphs},
contains a copy of $C_{2k+1}$ with at least one blue edge. 
\begin{figure}
\begin{center}
\includegraphics[scale=1,page=11]{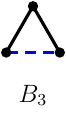}
\hskip 2cm 
\includegraphics[scale=1,page=12]{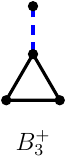}
\hskip 2cm 
\includegraphics[scale=1,page=13]{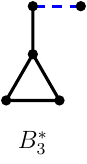}
\hskip 2cm 
\includegraphics[scale=1,page=14]{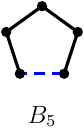}
\hskip 2cm 
\includegraphics[scale=1,page=15]{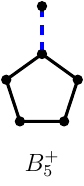}
\end{center}

\caption{The family of graphs $\cF_\Csev$ used in the construction of the $\cF_\Csev$-free flag algebra $\cA_\Csev$.}
\label{fig:FC7graphs}
\end{figure}

Let $\cF_\Csev := \left\{B_3,B_3^+,B_3^*,B_5,B_5^+\right\}$. As we have just observed, any sequence of graphs 
satisfying the assumptions of Theorem~\ref{thm:c7+c} is almost $\cF_\Csev$-free. We use the class $\cF_\Csev$
to construct the corresponding flag algebras. Again, we refer to them $\cA_\Csev$ and $\cA^\sigma_\Csev$ instead
of $\cA_{\cF_\Csev}$ and $\cA^\sigma_{\cF_\Csev}$.


Recall the three types $\lambda, \beta$ and $\rho$ depicted in Figure~\ref{fig:FC5graphs}.
We define $\cH^\lambda_4$, $\cH^\beta_4$ and $\cH^\rho_4$ to be the sets of all the $4$-vertex $\lambda$-flags in $\cA^\lambda_\Csev$,
$\beta$-flags in $\cA^\beta_\Csev$ and $\rho$-flags in $\cA^\rho_\Csev$, respectively.
Since the set of all the $4$-vertex flags in $\cA_\Csev$ is the same as the corresponding set
in $\cA_\Cfive$, we have again $\left|\cH^\lambda_4\right|=76$, $\left|\cH^\beta_4\right|=33$, and $\left|\cH^\rho_4\right|=43$.
Let $v_\lambda$, $v_\beta$ and $v_\rho$ be the appropriate vectors, where the $i$-th element of $v_\lambda$/$v_\beta$/$v_\rho$
is equal to the $i$-th element of $\cH^\lambda_4$/$\cH^\beta_4$/$\cH^\rho_4$.

This time, there are $741$ non-isomorphic red/blue-colored $\cF_\Csev$-free graphs on $6$ vertices.
With a slight abuse of notation, we again denote the set of all such graphs by $\cH_6$.
An application of the flag algebra method for $\cA_\Csev$ yields the following:
\begin{prop}\label{prop:c7+flag}
There exist three positive-semidefinite matrices $L$, $B$ and $R$
with rational entries and non-negative rational numbers $a$, $b$ and $c_H$, where $H \in \cH_6$,
such that in the algebra $\cA_\Csev$, the expression 
\[
  \unlab{v_\lambda^\T L v_\lambda}\lambda
+ \unlab{v_\beta^\T B v_\beta}\beta
+ \unlab{v_\rho^\T R v_\rho}\rho
+ \left(\figR \,+\, \figB \,-\, \figN  \right) \times \left(a\cdot\figEMP \,+\, b\cdot\figCOCh\right)
+ \sum\limits_{H \in \cH_6} c_H \cdot H
\]
is equal to
\[
\figTRG \,\times\,
\left(
9\cdot\figR - 4 \cdot \figV \,
\right)
\mbox{.}
\]
Moreover, if $H$ is a $6$-vertex red/blue-colored graph from $\cH_6$ that contains an induced copy of any $P_4 \in \cP_4$,
then $c_H > 0$.
\end{prop}

Again, we used CSDP and SAGE to find $L$, $B$, $R$, $a$, $b$ and $c_H$.
The webpage mentioned in Proposition~\ref{prop:c5flag} contains all the corresponding data,
as well as a short SAGE script that verifies the claimed identity.
As in Section~\ref{sec:c5}, the matrices $L$, $B$ and $R$ are decomposed as
\[
L = M_\lambda^\T \cdot \widehat{L} \cdot M_\lambda,\quad
B = M_\beta^\T \cdot \widehat{B} \cdot M_\beta,\quad \hbox{\rm{and}} \quad
R = M_\rho^\T \cdot \widehat{R} \cdot M_\rho,
\]
where $\widehat{L}$, $\widehat{B}$ and $\widehat{R}$ are positive definite matrices of sizes $58 \times 58$, $22 \times 22$ and $32\times 32$,
respectively, and $M_\lambda$, $M_\beta$ and $M_\rho$ are specific matrices of sizes $76\times 58$, $33\times 22$ and $43\times 32$, respectively.


\begin{cor}\label{cor:c7+flag}
If $\delta > 0$ and $\phi \in \Hom^+\left(\cA_{\Csev},\Real\right)$ such that
$\phi\left(\figR \,+\, \figB\right) \ge \frac12$ and $\phi\left(\figTRG\right) \ge \delta$,
then~$\phi\left(\figR\right) \ge \frac49$.
\end{cor}
\begin{proof} 
By Proposition~\ref{prop:c7+flag}, any $\psi \in \Hom^+\left(\cA_{\Csev},\Real\right)$ with $\psi\left(\figR \,+\, \figB\right) \ge 1/2$
must satisfy
\begin{equation}
\psi \left( \figTRG \,\times\, \left( 9\cdot\figR - 4 \right) \right) \ge 0
\mbox{.}
\label{eq:c7+flag}
\end{equation}
Since the triangle density in $\phi$ is at~least $\delta > 0$,
the density of red edges is at~least $4/9$.
\end{proof}

Now if Theorem~\ref{thm:c7+c} would be false, then there exists some absolute
constant $\eps_0 > 0$ and a convergent sequence of red/blue-colored almost
$\cF_\Csev$-free graphs $(G_i)_{i\in\Nat}$ of increasing orders $(n_i)$ such
that every $G_i$ has at most $(2/9 -\eps_0) \cdot \left(n_i\right)^2$ red
edges.  By Lemma~\ref{lem:stab}, the limit of the triangle densities in the
sequence must be positive.  However, Corollary~\ref{cor:c7+flag} yields that, for a
sufficiently large $i$, the graph $G_i$ has strictly more than $(2/9 -\eps_0)
\cdot \left(n_i\right)^2$ red edges; a contradiction.


\section{Stability of Constructions~\ref{cstn:cliquebip} and~\ref{cstn:c5}}
\label{sec:stability}

In this section, we show the corresponding stability for the extremal results
presented in Sections~\ref{sec:c5} and~\ref{sec:c7+}, and prove Theorems~\ref{thm:c5uniq} and~\ref{thm:c7+uniq}.
Let us start by recalling the following edge-colored variant of the induced graph removal lemma,
which is a direct consequence of~\cite[Theorem 1.5]{AusTao:2010}.
\begin{thm}
\label{thm:RL}
For any $\eps_{\RL}>0$ and a finite family of red/blue-colored graphs $\cF$, there exists $\delta_{\RL}>0$ such that
the following is true:
If $G$ is an $n$-vertex red/blue-colored graph with at most $\delta_{\RL} \cdot n^{v(F)}$ induced copies of $F$
for all $F \in \cF$, then the edge-set of $G$ can be modified on at most $\eps_{\RL}\cdot n^2$ pairs so that
no induced subgraph of the resulting graph is isomorphic to an element of $\cF$.
\end{thm}

Since the structure of Construction~\ref{cstn:cliquebip} is simpler than the structure of Construction~\ref{cstn:c5},
we begin with proving Theorem~\ref{thm:c7+uniq}.

\subsection{Odd cycles of length at least seven --- stability of Construction~\ref{cstn:cliquebip}} 

This whole subsection is devoted to the proof of Theorem~\ref{thm:c7+uniq}.
Recall that our task is, given an integer $k\ge3$ and $\eps > 0$, to find an integer $n_0$ and $\delta > 0$
so that for any graph $G$ with $n\ge n_0$ vertices and $\left(1/4 \pm \delta\right)n^2$ edges out of which $\left(2/9 \pm \delta\right)n^2$ occur in
$C_{2k+1}$, it holds that $G$ is $\eps n^2$-close in the edit-distance to Construction~\ref{cstn:cliquebip}.
Since Construction~\ref{cstn:cliquebip} is $O(n)$-close to a disjoint union of $2n/3$-vertex clique
and complete balanced bipartite graph on the remaining $n/3$ vertices, we show that $G$ is $\eps n^2$-close 
to this construction.
 
Fix such a graph $G$.
Following the notation from the previous sections, we color the edges of $G$ that occur in some
copy of $C_{2k+1}$ red, and the other edges of $G$ blue.
Since $G$ has only $\left(2/9 \pm \delta\right)n^2$ red edges, Lemma~\ref{lem:stab} yields 
that $G$ contains at least $\delta_k \cdot n^3$ triangles.
Without loss of generality, we may assume $\eps \ll \delta_k$.
Throughout the whole proof, we will use two auxiliary positive constants $\eps_{\RL}$ and $\delta_{\RL}$,
which will be determined during the proof, obeying the hierarchy
$\delta \ll \delta_{\RL} \ll \eps_{\RL} \ll \eps$.

By Corollary~\ref{cor:supersat}, we can choose $n_0$ to be a large enough integer 
such that the graph $G$ contains only $\delta n^{v(F)}$ copies of $F$ for all  $F \in \cF_\Csev$.
We continue our exposition by showing that $G$ cannot contain too many
induced paths on four vertices. 
Let $\cP_4$ be the set of all the six possible red/blue-colorings of the $4$-vertex path; see also Figure~\ref{fig:P4}.
The following lemma directly follows from Proposition~\ref{prop:c7+flag}.

\begin{figure}
\begin{center}
\hfill
\foreach \n in {22,...,27}{ \includegraphics[scale=0.95,page=\n]{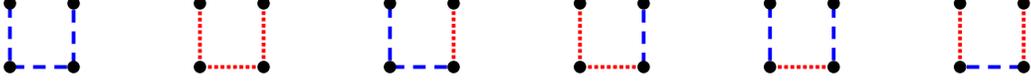}\hfill}
\hfill
\end{center}
\caption{The family $\cP_4$ containing all the $6$ non-isomorphic red/blue-colorings of $P_4$.}
\label{fig:P4}
\end{figure}

\begin{lem}
\label{lem:c7+uniq}
If $\phi \in \Hom^+\left(\cA_{\Csev},\Real\right)$ satisfies $\phi\left(\figR \right) = \frac49$ and $\phi\left(\figB\right) = \frac1{18}$,
then~$\phi\left(P\right) = 0$ for every $P \in \cP_4$.
\end{lem}

Let $n_0$ be large enough so that the limit identity proven by flag algebras
in Proposition~\ref{prop:c7+flag} holds with an error of order $O(\delta)$ for 
any graph in question with at least $n_0$ vertices.
Therefore, for any $F \in \cP_4$, it holds that $p(F,G) = O(\delta) \ll \delta_{\RL}$.
Set $\cF$ to be the family containing 
\begin{itemize}
\item all the red/blue-colored triangles with at least one blue edge, i.e, the graphs from $B_3$,
\item all the $4$-vertex red/blue-colored graphs that contain a copy of $B_3^+$,
\item all the $5$-vertex red/blue-colored graphs that contain a copy of $B_3^*$, and
\item the six elements of $\cP_4$.
\end{itemize}
Let $\delta_{\RL}$ be the constant from Theorem~\ref{thm:RL} applied with the constant $\eps_{\RL}$ and the family $\cF$.
Since $\delta \ll \delta_{\RL}$, the induced removal lemma yields a graph $G'$ with no induced copy of $F$ for all $F \in \cF$,
and differs from the original graph $G$ on at most $\eps_{\RL}\cdot n^2$ pairs.
In other words, $G'$ contains no induced path on $4$ vertices and no (not necessarily induced) copy
of $B_3$, $B_3^+$ or $B_3^*$. It follows that the number of edges in $G'$ is $\left(1/4 \pm 2\eps_{\RL}\right) n^2$.
By choosing $\eps_{\RL}$ to be much smaller than $\eps$, it is enough to
show that $G'$ is $\left(\eps/2 \cdot n^2\right)$-close to
Construction~\ref{cstn:cliquebip}.

Let $B$ be the set of vertices of $G'$ that are incident to at least one blue edge
or have a neighbor that is incident to a blue edge,
$H$ the subgraph induced by $B$, and $A$ the vertices of $G'$ that are not in $B$.
Now we prove the following three claims describing the structure of $G'$ in terms of $A$ and $B$.

\begin{claim}
\label{cl:stab7bipH}
The graph $H$ is bipartite.
\end{claim}
\begin{proof}
Suppose for contradiction there is an odd cycle in $H$. Since $H$ does not
contain any induced path on four vertices, $H$ must contain a triangle $xyz$.
Since $H$ is $B_3$-free, all three edges of the triangle must be red.
Also, neither $x$ nor $y$ nor $z$ is incident to a blue edge, because $H$ is $B_3^+$-free.
However, any of the three vertices, say $x$, is incident to a vertex $w$ such that $w$ is then
incident to a blue edge so $H$ fails to be $B_3^*$-free; a contradiction.
\end{proof}

\begin{claim}
\label{cl:stab7sizeA}
$A$ has size at least $\delta_k \cdot n$.
\end{claim}
\begin{proof}
The number of triangles in $G'$ is at least $\left(\delta_k - \eps_\RL \right) n^3$.
Since triangles in $G'$ can lie only inside the set $A$, $|A| \ge \left(\delta_k - \eps_\RL \right)^{1/3} \cdot n > \delta_k \cdot n$.
\end{proof}

\begin{claim}
\label{cl:stab7noABedge}
There are no edges between $A$ and $B$.
\end{claim}
\begin{proof}
Suppose there is an edge connecting two vertices $u \in A$ and $v \in B$.
Since $u \notin B$, the vertex $v$ is not incident to any blue edge, however,
it has a neighbor $w$, which then has a neighbor $x$ such that the edge $\{w,x\}$ has
blue color. Since $H$ is bipartite, $x$ is not connected to $v$. Hence
$\{u,v,w,x\}$ induces a $4$-vertex path in $G'$, which is a contradiction.
\end{proof}

Let $a := |A|/n$ and $b := 1-a = |B|/n$. In Claims~\ref{cl:stab7bipH}-\ref{cl:stab7noABedge}, we have shown
that the set $B$ induces a bipartite graph and there are no edges in $G'$ between $A$ and $B$.
Therefore, the edge-density of $G'$ is bounded by a function $f(a) := a^2 + (1-a)^2/2$.
The following observation directly follows from continuity of $f$, $f$ having
no local maximum on $(0,1)$, and compactness of $[0,1]$.
\begin{obs}
\label{obs:c7+stabopt}
The function $f(a)$ for $a \in [\delta_k, 2/3]$ 
has a unique maximum $1/2$ for $a=2/3$. Moreover, if the value of the function
for $a \in \left[\delta_k , 2/3+O(\eps_{\RL})\right]$ is close to $1/2$, then the value of $a$ is close to $2/3$.
\end{obs}

Since the number of edges of $G'$ is $\left(1/4 \pm 2\eps_{\RL}\right) n^2$, Observation~\ref{obs:c7+stabopt} yields that $a$ must be close to $2/3$.
It follows that $|A| = \left(2/3 \pm O(\eps_{\RL})\right) n$ and $|B| = \left(1/3 \pm O(\eps_\RL)\right) n$.
Moreover, the bipartite graph $H$ must have parts of sizes $\left(1/6 \pm O(\eps_\RL)\right) n$, and all but $O(\eps_\RL) n^2$ pairs between
the parts are joined by an edge. Finally, the number of non-adjacent pairs with both endpoints in $A$ is at most $O(\eps_\RL) n^2$.
Since $\eps_\RL \ll \eps$, we can easily modify $\eps/2 \cdot n^2$ pairs of $G'$ in order to obtain Construction~\ref{cstn:cliquebip}.
This finishes the proof of Theorem~\ref{thm:c7+uniq}.


\subsection{The pentagon case --- stability of Construction~\ref{cstn:c5}} 

We proceed very similarly as in the proof of Theorem~\ref{thm:c7+uniq}, but this
time, the arguments are tailored to Construction~\ref{cstn:c5}.
The graph $G$ has less than $0.22n^2$ red edges so Lemma~\ref{lem:stab} yields 
existence of at least $\delta_2 \cdot n^3$ triangles in $G$.
Without loss of generality, we may assume $\eps \ll \delta_2$.
As in the previous subsection, we will use two constants $\eps_{\RL} > 0$ and $\delta_{\RL} > 0$,
and we assume they obey the hierarchy
$\delta \ll \delta_{\RL} \ll \eps_{\RL} \ll \eps$.

Let $C_4^X$ be the red/blue-colored $4$-cycle with exactly one blue edge $\{u,v\}$ and a pendant red edge adjacent neither to $u$ nor to $v$,
and $\cP_5$ the set of all the ten possible red/blue-colorings of the $5$-vertex path; see also Figure~\ref{fig:C4X:P5}.
The moreover part of Proposition~\ref{prop:c5flag} yields the following lemma.

\begin{figure}
\begin{center}
\includegraphics[scale=0.9,page=28]{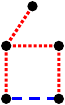}
\foreach \n in {29,...,38}{ \hfill \includegraphics[scale=0.9,page=\n]{EiC-fig}}
\end{center}
\caption{The red/blue-colored graph $C_4^X$ and the family $\cP_5$ containing all the $10$ non-isomorphic red/blue-colorings of $P_5$.}
\label{fig:C4X:P5}
\end{figure}

\begin{lem}\label{lem:c5uniq}
If $\phi \in \Hom^+\left(\cA_{\Cfive},\Real\right)$ that satisfy
$\phi\left(\figR \right) = \frac{\left(2+\sqrt{2}\right)}8$ and $\phi\left(\figB\right) = \frac{2-\sqrt{2}}8$,
then $\phi\left(P\right) = 0$ for every $P \in \left\{C_4^X\right\} \cup \cP_5$.
\end{lem}


Let $n_0$ be a large enough integer so that the graph $G$ contains only $\delta n^{v(F)}$ copies of $F$ for all  $F \in \cF_\Cfive$,
and the limit identity proven by flag algebras in Proposition~\ref{prop:c5flag} holds with an error of order $O(\delta)$ for 
any graph with at least $n_0$ vertices.
Set $\cF$ to be the family containing 
\begin{itemize}
\item all the red/blue-colored triangles with at least one blue edge, 
\item all the $4$-vertex red/blue-colored graphs that contain a copy of $B_3^+$,
\item all the $5$-vertex red/blue-colored graphs that contain a copy of $B_5$, 
\item the red/blue-colored graph $C_4^X$ and the ten elements of $\cP_5$.
\end{itemize}
Let $\delta_{\RL}$ be the constant from Theorem~\ref{thm:RL} applied with the constant $\eps_{\RL}$ and the family $\cF$.
Since $\delta \ll \delta_{\RL}$, by induced removal lemma there is a graph $G'$ differing from 
$G$ on at most $\eps_{\RL}\cdot n^2$ pairs that has no induced copy of $F$ for all $F \in \cF$.
Clearly, $G'$ has $\left(1/4 \pm 2\eps_{\RL}\right) n^2$ edges.
It remains to show that $G'$ is $\left(\eps/2 \cdot n^2\right)$-close to
Construction~\ref{cstn:c5}.

We begin with partitioning the vertices of $G'$ into three parts $X,Y,Z$ based
on their distance to vertices incident to blue edges.
Let $X$ be the set of vertices of $G'$ that are incident to at least one blue edge,
$Y$ the vertices that are incident only to red edges and have at least one neighbor
in $X$, and $Z$ the vertices of $G'$ that are neither in $X$ nor in $Y$.
We define $H$ to be the subgraph of $G'$ induced by $X \cup Y$.
Furthermore, let $X_0 \subseteq X$ be the set of all the vertices $x$
such that the connected component of $G'$ containing $x$ contains no
vertex from $Z$. Analogously, $Y_0 \subseteq Y$ are all the vertices
such that their connected component does not contain any vertex from $Z$.
Set $X_1 := X\setminus X_0$ and $Y_1 := Y\setminus Y_0$.
Having in mind the aim is to prove that $G'$ is close to Construction~\ref{cstn:c5},
we proceed with the following series of claims that describe the structure of $G'$.

\begin{claim}\label{cl:stab5bipH}
The graph $H$ is bipartite.
\end{claim}
\begin{proof}
Suppose for contradiction $H$ contains an odd cycle. Since $H$ does not
contain any induced $P_5$, $H$ contains either a triangle, or an induced pentagon.
In both cases, all the edges of the cycle must be red.

First, suppose that $H$ contains a triangle $u,v,w$. If at least one of the
three vertices is incident to a blue edge, we would have found a copy of $B_3^+$,
which is not possible. Therefore, $\{u,v,w\} \subseteq Y$.
Let $x_u \in X$ be a neighbor of $u$. If $v$ would be a neighbor of $x_u$ as well,
then $u,v,x_u$ and a blue edge going out from $x_u$ would create a copy of $B_3^+$.
Therefore, $\{x_u,v\}$ is not an edge. By the same reasoning, 
$\{x_u,w\}$ is not an edge and the vertex $v$ has neighbor
$x_v \in X$ such that neither $\{x_v,u\}$ nor $\{x_v,w\}$ are edges.
Now let $x \in X$ be a vertex connected to $x_u$ by a blue edge.
Since $H$ is $B_3$-free, $x$ is not a neighbor of $u$,
and since $H$ is $B_5$-free, $x$ is neither a neighbor of $v$ nor $x_v$ nor $w$.
The path $x,x_u,u,v,x_v$ cannot be induced and therefore there is an edge
between $x_u$ and $x_v$. But then the vertices $w,v,x_v,x_u,x$ span an induced
$P_5$, which is a contradiction.
For the rest of the proof, we will assume that $H$ is triangle-free.

Now suppose $H$ has an induced pentagon $u_1,u_2,u_3,u_4,u_5$
so that one of its vertices, say $u_1$, is incident to a blue edge.
Let $x_1 \in X$ be one of the neighbors of $u_1$ that is joined to $u_1$ by a
blue edge. If $x_1$ would be joined by an edge either to $u_2$ or $u_5$, then
we have found a copy of $B_3$.  Since $H$ is also $B_5$-free, the vertex $x_1$
cannot be joined by an edge to $u_3$ or $u_4$. Therefore, $x_1,u_1,u_2,u_3,u_4$
is an induced path of length four, a contradiction. 

Finally, suppose there is an induced pentagon $u_1,u_2,u_3,u_4,u_5$ such that all
the edges incident to the five vertices are red. The vertex $u_1$ must have a neighbor,
say $x_1$, that is incident to a blue edge. We already know that $H$ is triangle-free,
so $x_1$ is adjacent neither to $u_2$, nor to $u_5$. Also, if $x_1$ would be a neighbor
of $u_3$, then $u_1,x_1,u_3,u_4,u_5$ is a $5$-cycle with one endpoint incident to a blue
edge, which we already excluded in the previous paragraph. Analogously, $x_1$ is not adjacent
to $u_4$, and hence $x_1,u_1,u_2,u_3,u_4$ is an induced path of length four; a contradiction.
\end{proof}

\begin{claim}
$Z$ has size at least $\delta_2/2 \cdot n$.
\end{claim}
\begin{proof}
As in Claim~\ref{cl:stab7sizeA}, the number of triangles in $G'$ is at least $\left(\delta_2 - \eps_\RL \right) n^3$.
Since every triangle has at least one vertex in $Z$, $|Z| \ge \left(\delta_2 - \eps_\RL \right) \cdot n > \delta_2/2 \cdot n$.
\end{proof}

Let $H_1$ be the subgraph of $H$ induced by $X_1$. We continue in our exposition and find a good bipartition of $H_1$.

\begin{claim}
\label{cl:noblueedge}
If $u$ and $v$ are two vertices from $X_1$ that are joined by a blue edge, then at most one of the two vertices has a neighbor in $Y_1$.
\end{claim}
\begin{proof}
Suppose for contradiction there are two such vertices $u$ and $v$.
Since $H$ is bipartite and has no induced $P_5$, at least one of the two vertices
is within distance exactly two to a vertex $z \in Z$. Without loss of generality,
let $u$ be the vertex, and let $y_u$ be the middle vertex on a shortest path
between $u$ and $z$.

Let $y_v \in Y_1$ be a neighbor of $v$.
Since $H$ is bipartite, neither $y_u$ is a neighbor of $v$, nor $y_v$ is a neighbor of $u$.
Also, $G'$ is $B_5$-free, hence the vertex $z$ is not a neighbor of $y_v$,
and by definition, there are no edges between $Z$ and $X_1$.
So either $y_u$ and $y_v$ are not joined by an edge and $y_v,v,u,y_u,z$
induces a path, which contradicts that $G'$ does not contain an induced $P_5$. Or, $\{y_u,y_v\}$ is an edge, but then the vertices
induces $C_4^X$; a contradiction.
\end{proof}

\begin{claim}
Let $u\in X_1$ and $v \in X_1$ be two vertices from the same connected component of $H_1$.
If both $u$ and $v$ have a neighbor in $Y_1$,
then there exists a vertex $w \in V(H_1)$ such that both $\{u,w\}$ and $\{w,v\}$ are edges in $H_1$.
\end{claim}
\begin{proof}
Analogously to the previous claim, we may assume that one of the two vertices, say $u$,
has a neighbor $y \in Y_1$ such that $y$ is adjacent to a vertex $z \in Z$.
On the other hand, since $v \in X_1$, it must have a neighbor $t \in X_1$ such that $\{t,v\}$ is blue.
By Claim~\ref{cl:noblueedge}, $t \neq u$.
If $\{t,u\}$ or $\{v,y\}$ is an edge, we are done by letting $w:=t$ or $w:=y$, respectively.
For the rest of the proof, we assume that neither $\{t,u\}$ nor $\{v,y\}$ is an edge.
Also, Claim~\ref{cl:noblueedge} yields that $t$ has no neighbor in $Y_1$, so in particular,
$\{t,y\}$ is not an edge.

Now we show that $u$ is not adjacent to $v$.
Suppose there is an edge between $u$ and $v$. By Claim~\ref{cl:noblueedge}, the edge must be red.
Recall that the vertex $y$ has a neighbor $z \in Z$.
There are no edges between $Z$ and $X_1$ so the vertices $t,v,u,y,z$ induces $P_5$, which is a contradiction.

Suppose there is no $w \in V(H_1)$ such that $u,w,v$ is a path of length two.
Since $H_1$ does not contain any induced path of length four, there exist vertices $x_u \in V(H_1)$ and $x_v \in V(H_1)$
such that $u,x_u,x_v,v$ is a path of length three. 
The vertex $t$ must be connected to $x_u$,
as otherwise $u,x_u,x_v,v,t$ is an induced $P_5$.
However, $H$ is bipartite so the path $v,t,x_u,u,y$ must be induced; a contradiction.
\end{proof}

The last claim immediately yields the following corollary.
\begin{cor}
There exists a partition of the set $X_1$ into two parts $A_1$ and $B_1$ such
that both $A_1$ and $B_1$ are independent sets in $G'$, and there are no edges
between $A_1$ and $Y_1$.
\end{cor}

This also implies that the set $Y_1$ must be independent.
\begin{claim}\label{cl:stab5indepY1}
The set $Y_1$ is an independent set in $G'$.
\end{claim}
\begin{proof}
Suppose there is an edge between two vertices $u \in Y_1$ and $v \in Y_1$.
By definition, there exist two vertices $b_u \in B_1$ and $b_v \in B_1$ that are adjacent
to $u$ and $v$, respectively. The two vertices are distinct and none of them
can be adjacent to both $u$ and $v$.
Let $a \in A_1$ be a neighbor of $b_u$ such that $\{a,b_u\}$ is a blue edge.
The vertex $a$ cannot be adjacent to $b_v$, which yields that $a,b_u,u,v,b_v$
is an induced path of length four; a contradiction.
\end{proof}

Now let $(A_0,B_0)$ be the color classes of an arbitrary $2$-coloring of the
bipartite graph induced by $X_0 \cup Y_0$.  We define the following four sets
that partition the set $V(G')$:
$A:= A_0 \cup A_1$, $B:= B_0 \cup B_1$, $C:= Y_1$, and $D:=Z$.
Claims~\ref{cl:stab5bipH}-\ref{cl:stab5indepY1} yield that $G'$ must have the following structure.

\begin{cor}
$\{A,B,C,D\}$ is a partition of the vertex-set of $G'$,
the sets $A$, $B$ and $C$ are independent sets in $G'$, and every edge $e$ of
$G'$ goes either between $A$ and $B$, or $B$ and $C$, or $C$ and $D$,
or inside $D$. Moreover, if $e$ is blue, then $e$ must go between $A$ and $B$.
\end{cor}

Let $a:=|A|/n$, $b:=|B|/n$, $c:=|C|/n$ and $d:=|D|/n$.
The edge-density of $G'$ can be upper-bounded by $f(a,b,c,d) := 2ab+2bc+2cd+d^2$.
Let us analyze the maximum value of $f$ under constraints on $a$, $b$, $c$ and $d$
that we have already established. All of that is summarized in the following optimization
problem:

\begin{alignat*}{2}
    \textbf{maximize: }    & 2ab+2bc+2cd+d^2  \\
    \textbf{subject to: } & a\ge0, \quad b\ge0, \quad c \ge 0, \\
    & d=1-a-b-c, \\
    & ab \ge \left(2 - \sqrt2\right)/16,\\
    & d \ge  \delta_2/2 .\\
\end{alignat*}

Clearly, if the values of $a$, $b$, $c$ and $d$ are equal to those coming from Construction~\ref{cstn:c5},
then $f(a,b,c,d) = 1/2$. The following proposition shows that there is no other
point $(a',b',c',d') \in \Real^4$ that would satisfy the constraints and also attain the value $1/2$.

\begin{claim}\label{cl:stab5opt}
The optimization problem has a unique solution at the point
\[\left(a_m,b_m,c_m,d_m\right) = \left(\frac12 - \frac{\sqrt2}4, \frac14, \frac14, \frac{\sqrt2}4 \right),\]
Moreover, if a point $(a',b',c',d')$ satisfies all the constraints and $f(a',b',c',d')$
is close to $1/2$, then $(a',b',c',d')$ is close to $(a_m,b_m,c_m,d_m)$.
\end{claim}

This claim immediately yields that $G'$ is close to Construction~\ref{cstn:c5}, hence proving
the claim finishes the proof of Theorem~\ref{thm:c5uniq}

\begin{proof}[Proof of Claim~\ref{cl:stab5opt}]
Let $\left(a_0,b_0,c_0,d_0\right)\in \Real^4$ be a point that satisfies the constraints and maximizes the objective function.
In particular, $f\left(a_0,b_0,c_0,d_0\right)\ge 1/2$.

First, we show that $a_0b_0 = \left(2 - \sqrt2\right)/16$.
If $a_0b_0 > \left(2 - \sqrt2\right)/16$, then let $\alpha:=a_0 - \frac{2-\sqrt2}{16\cdot b_0}$.
It follows that \[f(a_0-\alpha,b_0,c_0+\alpha,d_0) = f\left(a_0,b_0,c_0,d_0\right) + \alpha \cdot d_0,\]
and the point $(a_0,b_0,c_0,d_0)$ was not an optimal solution.
Next, we bound $b_0$ away from $1/2$. Observe that
\[
f\left(a,b,c,d\right) = 2ab + 2(b+d)(1-a-b-d) + {d}^2
= 2b(1-b)  - {d}^2 + 2d\left(1-a-2b\right) .
\]
Therefore, \[
\left(1-a_0-2b_0\right) = \left(1-\frac{2-\sqrt2}{16\cdot b_0}-2b_0\right) > 0,
\]
as otherwise $f(a_0,b_0,c_0,d_0) \le 1/2 - {d_0}^2 \le 1/2 - (\delta_2/2)^2 < 1/2$ contradicting $f(a_0,b_0,c_0,d_0) \ge 1/2$.
Thus, \[b_0 < \frac{2+2^{3/4}}8 < 0.4603.\]
Moreover, the maximum value of $2b_0\cdot(1-b_0)$ is at most $\left(6-\sqrt{2}+2^{7/4}\right)/16 < 0.497$. On the other hand, the maximum
value of $\left(1-\frac{2-\sqrt2}{16\cdot b_0}-2b_0\right)$ is at most $1-\sqrt{1-1/\sqrt{2}}<0.46$.
Since $f(a_0,b_0,c_0,d_0) \ge 1/2$, we conclude that $d_0 > 0.003$.
%
%

Suppose now that $b_0 \neq c_0$. If $b_0 < c_0$,
then \[
 f\left(a_0,b_0,c_0-\alpha,d_0+\alpha\right) - f\left(a_0,b_0,c_0,d_0\right) = 2(c_0-b_0)\alpha + \alpha^2
,\] where $\alpha = c_0-b_0$; a contradiction. 
On the other hand, if $c_0 < b_0$, then  \[
 f\left(a_0,b_0,c_0+\alpha,d_0-\alpha\right) - f\left(a_0,b_0,c_0,d_0\right) = 2(b_0-c_0)\alpha - \alpha^2 \ge \alpha^2
,\] where this time $\alpha = \min(b_0 - c_0,d_0 - 0.003)$.
We conclude that
\begin{equation}\label{eq:stab5finalf}
f\left(a_0,b_0,c_0,d_0\right) = \frac{2-\sqrt2}8 + 2{c_0}^2
+ 2c_0 \cdot \left(1-\frac{2-\sqrt2}{16\cdot c_0}-2c_0\right)
+ \left(1-\frac{2-\sqrt2}{16\cdot c_0}-2c_0\right)^2
.
\end{equation}



Since swapping the values of $a_0$ and $b_0$ changes the objective function by $c_0(a_0 - b_0)$, it holds that $b_0 \ge a_0$.
In particular, $c_0 = b_0 \ge \left(\sqrt{2-\sqrt2}\right)/4 > 0.19$. The right-hand side of (\ref{eq:stab5finalf})
depends only on $c_0$ and $c_0 \in [0.19,0.4603]$. It is straightforward to
check that in this range, the value of (\ref{eq:stab5finalf}) is at most $1/2$,
and the unique point where the value is attained is $c_0 = 1/4$. Therefore,
$b_0 = 1/4$, $a_0 = \frac{2 - \sqrt2}4$ and $d_0 = \frac{\sqrt2}4$.
By continuity of $f(a,b,c,d)$ and compactness of $[0,1]^4$, it also follows that if $f(a',b',c',d')$ is close to $1/2$,
then $(a',b',c',d')$ is close to $\left(a_0,b_0,c_0,d_0\right)$.
\end{proof}
 

\section{Exact result for pentagons}
\label{sec:c5exact}

For a graph $G$, we define $\cC_5(G)$ to be the set of all edges of $G$
that occur in a copy of $C_5$ in $G$. In other words,
\[\cC_5(G) = \bigcup\limits_{H \subseteq G, H \cong C_5} E(H) .\]
Let $\cE_n$ be the set of all $n$-vertex graphs with exactly $\lfloor
n^2/4\rfloor +1$ edges, and define 
\[F(n) := \min\limits_{G \in \cE_n}|\cC_5(G)|.\]
For convenience, we set $\widetilde{F}(n) := \lfloor n^2/4\rfloor +1 - F(n)$.

Next, let $\cE'_n$ be the set of all $n$-vertex graphs with
at least $\lfloor n^2/4\rfloor +1$ edges.  It immediately follows that for any $G \in
\cE'_n$, it holds $|\cC_5(G)| \ge F(n)$, and if $|\cC_5(G)| = F(n)$, then $G \in
\cE_n$.
Finally we define $\cG_n \subseteq \cE'_n$ to be the set of all $G \in \cE'_n$ with $|\cC_5(G)| = F(n)$.

We call a quadruple of non-negative integers $(a,b,c,d)$ $n$-extremal,
if the following is satisfied:
\begin{itemize}
\item $a+b+c+d=n$,
\item $a \cdot b = \widetilde{F}(n)$, and
\item $a \cdot b + b\cdot c + c \cdot d + \binom{d}{2} > \frac{n^2}4$.
\end{itemize}

The main theorem of this section is the following:
\begin{theorem}
\label{thm:c5exact}
There exists an integer $n_0$ such that the following holds for any $n \ge n_0$.
If $G \in \cG_n$, then $V(G)$ can be partitioned 
into four sets $A$, $B$, $C$ and $D$ such that
\begin{itemize}
\item the quadruple $(|A|,|B|,|C|,|D|)$ is $n$-extremal,
\item $A$, $B$ and $C$ are independent sets of $G$,
\item $\{u,v\} \in E(G)$ for any $u \in A$ and $v \in B$,
\item $\{u,v\} \notin E(G)$ for any $u \in A$  and $v \in C \cup D$, and
\item $\{u,v\} \notin E(G)$ for any $u \in B$  and $v \in D$.
\end{itemize}
\end{theorem}

An immediate consequence of this theorem is
that $(a,b,c,d)$ is $n$-extremal if and only if it solves
the following integer quadratic program:

\begin{alignat*}{2}
    \textbf{maximize: }    & a \cdot b \\
    \textbf{subject to: } & a\in \Nat, \quad b\in\Nat, \quad c \in \Nat, \quad d \in \Nat, \\
    & a \cdot b + b\cdot c + c \cdot d + \binom{d}{2} > \frac{n^2}4,\\
    & a+b+c+d = n. \\
\end{alignat*}

Since the exact solution of this maximization problem for a given integer $n$
depends on errors in rounding expressions like $\sqrt{2}n/4$, we leave it in
this form. Approximate values of $a$, $b$, $c$ and $d$ are indeed given by
Construction~\ref{cstn:c5}.

\begin{proof}[Proof of Theorem~\ref{thm:c5exact}]
Theorems~\ref{thm:c5} and~\ref{thm:c5uniq} immediately yield that
for any $\eps > 0$, there exists an integer $n_0$ so that if $n \ge n_0$,
then by adding or removing $\eps n^2$ edges in a graph $G \in \cG_n$
we obtain the graph from Construction~\ref{cstn:c5}.
Moreover, the value of $n_0$ will be large enough so that Construction~\ref{cstn:c5}
yields that $F(n) = ((2+\sqrt{2})/16\pm\eps)n^2$ 
and $\widetilde{F}(n) = ((2-\sqrt{2})/16\pm\eps)n^2$ for every $n \ge n_0$.

Fix an integer $n \ge n_0$ and any graph $G \in \cG_n$, and let $V$ be the vertex-set of $G$.
Clearly, for any $\eps' > 0$, we can find $\eps > 0$ small enough
so that $V$ can be partitioned into five sets $A_0, B_0, C_0, D_0$ and $X$ such that
\begin{itemize}
\item $|A_0| = (1/2-\sqrt{2}/4 \pm \eps')\cdot n$,
$|B_0| = (1/4 \pm \eps') \cdot n$,
$|C_0| = (1/4 \pm \eps') \cdot n$,
$|D_0| = (\sqrt{2}/4 \pm \eps') \cdot n$,
\item $0 \le |X| \le \eps' \cdot n$,
\item $\deg_{A_0}(u) \ge (1-\eps')|A_0|$ for every $u \in B_0$,
\item $\deg_{B_0}(u) \ge (1-\eps')|B_0|$ for every $u \in A_0 \cup C_0$,
\item $\deg_{C_0}(u) \ge (1-\eps')|C_0|$ for every $u \in B_0 \cup D_0$,
\item $\deg_{D_0}(u) \ge (1-\eps')|D_0|$ for every $u \in C_0$, and
\item the induced subgraph $G[D_0]$ has edge-density at least $1-\eps'$.
\end{itemize}
In other words, the stability result from Section~\ref{sec:stability} yields
an approximate structure of $G$. In the following series of claims, we will
show that the extremality of $G$ allows us to ``clean up'' this description to
the one claimed in the statement of the theorem.
For the rest of the proof, we will assume $\eps' > 0$ is sufficiently small ($\eps' < 10^{-4}$ would be sufficient).

For a set $S \subseteq V$, we denote by $E(S)$ the set of edges of the subgraph
induced by $S$, i.e., $E(S) = E(G[S])$.  For two disjoint $X, Y \subseteq V$,
we denote by $E(X,Y)$ the set of edges in $G$ with exactly one endpoint in $X$
and the other endpoint in $Y$.

First, let us observe that every graph with more than $n^2/4$ edges has at most $\widetilde{F}(n)$ edges that
do not occur in $C_5$.
\begin{claim}
\label{cl:c5duality}
There is no $n$-vertex graph $G' \in \cE'_n$ with $|E(G') \setminus \cC_5(G')| > \widetilde{F}(n)$.
\end{claim}
\begin{proof}
Indeed, otherwise remove from $G'$ arbitrarily chosen $|E(G')| - \lfloor n^2/4 \rfloor - 1$ edges in $\cC_5(G')$.
The obtained graph has less than $\lfloor n^2/4 \rfloor + 1 - \widetilde{F}(n) = F(n)$ edges that occur in $C_5$,
a contradiction.
\end{proof}

We continue with three simple claims that all the edges between the parts $B_0$ and $C_0$, $C_0$ and $D_0$,
and inside $D_0$ occur in a copy of $C_5$.
\begin{claim}
$E(B_0,C_0) \subseteq \cC_5(G)$.
\end{claim}
\begin{proof}
Fix any $\{u,v\} \in E(B_0,C_0)$ with $u \in B_0$.
Let $v' \in C_0$ be a neighbor of $u$ in $C_0$ different from $v$.
Since both $v$ and $v'$ have more than $|D_0|/2$ neighbors in $D_0$ and the edge-density of $G[D_0]$ is $(1-\eps')$,
there exist a vertex $w \in D_0$ connected to $v$, and vertex $w' \in D_0$ connected to $v'$ such that $\{w,w'\}$ is an edge. 
Therefore, $uvww'v'$ forms a $C_5$ in $G$.
\end{proof}

\begin{claim}
$E(C_0,D_0) \subseteq \cC_5(G)$.
\end{claim}
\begin{proof}
Fix any $\{u,v\} \in E(C_0,D_0)$ with $u \in C_0$.
Let $v' \in D_0$ be a neighbor of $u$ in $D_0$ different from $v$.
Since the edge-density of $G[D_0]$ is $(1-\eps')$, there
exists a path of length three between $v$ and $v'$ in $G[D_0]$.
This path together with the edges $\{u,v\}$ and $\{u,v'\}$ forms a $C_5$.
\end{proof}

\begin{claim}
$E(G[D_0]) \subseteq \cC_5(G)$.
\end{claim}
\begin{proof}
Let $u$ and $v$ be two adjacent vertices from $D_0$.
Since the edge-density of $G[D_0]$ is $(1-\eps')$, there
is a path of length four between $u$ and $w$,
which together with $\{u,v\}$ forms a $C_5$.
\end{proof}

Since $\left| E(B_0,C_0) \cup E(C_0,D_0) \cup E(G[D_0]) \right| \ge (2+\sqrt{2} - 3\eps')n^2$,
we immediately conclude that
\begin{cor}
\label{cor:c5a0b0edge}
$\left| E(A_0,B_0) \cap \cC_5(G) \right| < 5\eps' n^2$.
\end{cor}

Let $E' := E(A_0,B_0) \cup E(B_0,C_0) \cup E(C_0,D_0) \cup E(G[D_0])$.
Since most of the edges between $A_0$ and $B_0$ do not occur in any $C_5$,
we now get much better control on the edges in $E \setminus E'$.
We start with the following two claims.
\begin{claim}
\label{cl:c5noAvB}
There is no vertex $z \in V$ adjacent both to $u \in A_0$ and $v \in B_0$.
\end{claim}
\begin{proof}
Suppose for contradiction there is such a vertex $z$, and let $u \in A_0$
and $v \in B_0$ be its two neighbors. Let $v' \in B_0$ be a neigbor of $u$
different from $v$. Clearly, there are at least $(1-\eps')|A_0|$ ways of
choosing $v'$. The vertices $v$ and $v'$ have at least $\deg_{A_0}(v) +
\deg_{A_0}(v') - |A_0| \ge (1-2\eps')|A_0|$ common neigbors $u' \in A_0$. Each
such choice of $u'$ and $v'$ yields a copy of $C_5$ on the vertices $uzvu'v'$.
In particular the edge $\{u',v'\} \in E(A_0,B_0) \cap \cC_5(G)$. However,
there are at least $(1-3\eps')|A_0||B_0| > 0.03n^2$ choices of $\{u',v'\}$,
which contradicts Corollary~\ref{cor:c5a0b0edge}.
\end{proof}

\begin{claim}
\label{cl:c5noBvC}
There is no vertex $z \in V$ adjacent both to $u \in B_0$ and $v \in C_0$.
\end{claim}
\begin{proof}
Suppose not, and let $u \in B_0$ and $v \in C_0$ be two neighbors of $z$.
Let $u' \in B_0$ be any of the $(1-\eps')|B_0|$ neighbors of $v$ different from $u$.
Since $\deg_{A_0}(u) + \deg_{A_0}(u') - |A_0| \ge (1-2\eps')|A_0|$,
there are at least $(1-2\eps')|A_0| \cdot (1-\eps')|B_0| > (1-3\eps')|A_0||B_0|$
edges from $E(A_0,B_0)$ that occur in a $C_5$ (note that $uzvu'w$, where $w \in A_0$ is a common
neighbor of $u$ and $u'$, form a $C_5$); a contradiction.
\end{proof}

A direct consequence of the last two claims is the following.
\begin{cor}
The sets $A_0$, $B_0$ and $C_0$ are independent,
and $|E(A_0,C_0)| = |E(B_0,D_0)| = 0$.
\end{cor}

Now move our attention to paths of length at most two between $A_0$ and $D_0$.
Let $Y \subseteq A_0$ be the set of vertices $u \in A_0$ such that
there exist vertices $v \in V$ and $w \in D_0$ such that both 
$\{u,v\} \in E(G)$ and $\{v,w\} \in E(G)$.
\begin{claim}
$|Y| < 21 \eps' n$.
\end{claim}
\begin{proof}
For each edge $\{y,v\}$ with $y \in Y$ and $v \in B_0$, consider
vertices $z \in V\setminus\{y,v\}$ and $x \in N_{D_0}(z)$ such that
$yzx$ is a $3$-vertex path in $G$. Note that such a path exists by the
definition of $Y$. Since $|N_{C_0}(x) \cap N_{C_0}(v)| > |C_0|/2$, we conclude
that $\{y,v\} \in \cC_5(G)$. The edge $\{y,v\}$ can be chosen in at least $|Y|
\cdot (1-\eps') |B_0|$ ways, so by Corollary~\ref{cor:c5a0b0edge} we conclude that
\[  
|Y| \le \frac{5\eps' n^2}{(1-\eps')|B_0|} < \frac{20\eps' n}{1-2\eps'} < 21\eps'n
.\]
\end{proof}

We set $A_0' := A_0 \setminus Y$ and $Z := X \cup Y$.
We continue our exposition by establishing a lower bound on the minimum degree of $G$.
We start with the following claim.
\begin{claim}
\label{cl:c5clone}
There exists a vertex $u \in A_0'$ incident to at least $(1/4 - 191\eps')n$ edges
not in $\cC_5(G)$, and a vertex $u' \in B_0$ incident to at least
$((2-\sqrt{2})/4-93\eps')n$ edges that are not in $\cC_5(G)$.    
\end{claim}
\begin{proof}
There are at least $\widetilde{F}(n) - |Z|n \ge ((2-\sqrt{2})/16-23\eps')n^2$ edges between $A_0'$ and $B_0$
that do not occur in $C_5$. Since $|A_0'| \le ((2-\sqrt{2})/4+\eps')n$, 
there is a vertex $u \in A_0'$ incident to at least $(1/4 - 191\eps')n$ such edges.
Similarly, $|B_0| \le (1/4+\eps')n$, which implies existence of $u' \in B_0$ incident
to at least $((2-\sqrt{2})/4-93\eps')n$ edges in $E(G) \setminus \cC_5(G)$.
\end{proof}

\begin{claim}
\label{cl:c5mindeg}
For any $v \in V$, $\deg(v) \ge (1/4 - 191\eps')n$.
\end{claim}
\begin{proof}
Otherwise consider the graph $G'$ obtained from $G$ by removing the vertex $v$ and cloning the
vertex $u$ from the previous claim. $G'$ has more edges that do not occur in $C_5$ than $G$
and also $|E(G')| > |E(G)|$, a contradiction with Claim~\ref{cl:c5duality}.
\end{proof}

\begin{cor}
There is no vertex $z \in Z$ such that $N(z) \subseteq A'_0 \cup Z$.
\end{cor}
\begin{proof}
Indeed, any such $z$ would have 
\[\deg(z) \le |A'_0|+|Z|=|A_0|+|X| \le ((2 - \sqrt{2})/4 + 2\eps')n < 0.15n < (1/4-191\eps')n,\]
which contradicts the previous claim.
\end{proof}

Now we are ready to split the vertices $z \in Z$ based on their adjacencies to the vertices outside of $Z$.
First, let $Z' := \{z \in Z \cond \exists u \in B_0: \{z,u\} \in E(G) \}$.
Claims~\ref{cl:c5noAvB} and~\ref{cl:c5noBvC} yield that no vertex $z \in Z'$ has a neighbor in $A_0 \cup C_0$.
We define
\[C_1 := \{z \in Z' \cond \exists v\in V \land \exists w \in D_0: \{z,v\} \in E(G) \land \{v,w\} \in E(G) \},\]
and $A_1:=Z' \setminus C_1$. Note that $Y \subseteq C_1$, and if $z \in Z'$ has a neighbor in $D_0$, then $z \in C_1$.
Let us first focus on the set $A_1$.
\begin{claim}
For all $z \in A_1$, $\deg_{B_0}(z) > (1-214\eps')|B_0|$.
\end{claim}
\begin{proof}
If there exist a vertex $z\in A_1$ 
with $\deg_{B_0}(z) \le (1-214\eps')|B_0|$,
then its total degree in $G$ is at most
\[\deg_{B_0}(z) + \deg_Z(z) < (1-214\eps')(1/4+\eps')n + |Z| \le (1/4 - 191\eps') n,\]
a contradiction.
\end{proof}

\begin{cor}
$A_1$ is an independent set in $G$.
\end{cor}
\begin{proof}
If there is an edge in $A_1$, then this edge together with any edge in $E(A_0,B_0)$ are in a $C_5$, contradicts Corollary~\ref{cor:c5a0b0edge}.
\end{proof}

We set $A := A_0' \cup A_1$, and continue our exposition by analyzing the
vertices $B_1 := \{z \in Z \cond \exists u \in A: \{z,u\} \in E(G)\}$.  Note
that $B_1 \cap Z' = \emptyset$.
By Claim~\ref{cl:c5noAvB} and definitions of the sets $A_0'$ and $A_1$, 
we conclude that both $|E(B_1,B_0)| = 0$ and $|E(B_1,D_0)| = 0$.
In the following two claims, we study the edges between $B_1$ and $C_0 \cup A$.

\begin{claim}
\label{cl:c5degB1C0}
For every $v \in B_1$, $\deg_{C_0}(v) > n/10$.
\end{claim}
\begin{proof}
We know that $v \in B_1$ can be adjacent only to the vertices from $A_0 \cup C_0 \cup X$.
Therefore, Claim~\ref{cl:c5mindeg} yields that $v$ has at least
\[(1/4-191\eps')n - |A_0| - |X|  \ge (\sqrt{2} - 1 -193\eps')n > n/10 \]
neighbors in $C_0$.
\end{proof}

\begin{claim}
\label{cl:c5degB1A0}
For every $v \in B_1$, $\deg_A(v) \ge (1 - 800 \eps')|A|$.
\end{claim}
\begin{proof}
First observe that any edge $\{v,w\}$ with $w \in C_0$ is contained in $\cC_5(G)$.
Indeed, consider a neighbor $w' \in N_{C_0}(v) \setminus \{w\}$ and two
adjacent vertices $x \in N_{D_0}(w)$ and $x'\in N_{D_0}(w')$.

Suppose for contradiction $\deg_A(v) < (1 - 800 \eps')|A| < |A|-116\eps'n$. In particular, $v$ is incident
to at most
\[\deg_{A}(v)  + \deg_Z(v) < |A| - 94 \eps'n \le \frac{2-\sqrt2}4 -93 \eps'n \]
edges not in $\cC_5(G)$.
Let $u' \in B_0$ be the vertex from Claim~\ref{cl:c5clone} with at least $\left((2-\sqrt{2})/4-93\eps'\right)n$
incident edges that are not in $\cC_5(G)$. Moreover,
$\deg(u') \ge |A| + |C_0| - 2 \eps' n$.
Therefore, removing the vertex $v$ and adding a clone of the vertex $u'$ yield an
$n$-vertex graph with more than $n^2/4$ edges that contradicts Claim~\ref{cl:c5duality}.
\end{proof}

\begin{cor}
$B_1$ is an independent set of $G$.
\end{cor}
\begin{proof}
If there is an edge in $B_1$, then this edge together with any edge in $E(A_0,B_0)$ are in a $C_5$, contradicts Corollary~\ref{cor:c5a0b0edge}.
\end{proof}

Let $B:=B_0 \cup B_1$. Recall the definition of $C_1$, i.e.,
\[
C_1 = \{u \in Z \cond \deg_{B_0}(u) \ge 1 \land \exists v \in V, w \in D_0 : \{u,v\} \in E(G) \land \{v,w\} \in E(G) \}
.\]
By Claim~\ref{cl:c5noBvC}, there are no edges between $C_1$ and $C_0$, and by the definition of $B_1$,
there are no edges between $C_1$ and $A$. Let us now show that vertices $v \in C_1$ have 
many neighbors in $B_0$.

\begin{claim}
\label{cl:c5degC1B0}
For every $u \in C_1$, $\deg_{B_0}(u) \ge n/25$.
\end{claim}
\begin{proof}
As noted above, $u$ can be adjacent only to the vertices in $B_0 \cup D_0 \cup Z$.
First observe that for every $z \in N_{D_0}(u)$, the edge $\{u,z\} \in \cC_5(G)$.
Indeed, let $x \in N_{B_0}(u)$ and $y \in N_{C_0}(x)$ be chosen arbitrarily. Since $y \in C_0$
and $z \in D_0$, there exist a common neighbor of $y$ and $z$ which encloses a $C_5$.
Analogously, we show $\{u,x\} \in \cC_5(G)$ for every $x \in N_{B_0}(u)$.
Consider the vertices $v \in V$ and $w \in D_0$ with $\{u,v\} \in E(G)$ and $\{v,w\} \in E(G)$
witnessing that $u \in C_1$. Since $x \in B_0$ and $w \in D_0$, the two vertices must have a common
neighbor which yields $\{u,x\} \in \cC_5(G)$. 

The last paragraph shows that $u$ is incident to at most $|Z| < 22 \eps' n$
edges that do not occur in $C_5$. On the other hand,
$|Z\cup D_0| < (\sqrt{2}/4 + 23\eps')n$. So if $\deg_{B_0}(v) < n/25$, then
\[ \deg(v) < 0.396n < \deg(u'),\]
where $u' \in B_0$ is the vertex from Claim~\ref{cl:c5clone}.
Therefore, the graph obtained by removing the vertex $u$ and cloning the vertex $u'$ 
contradicts Claim~\ref{cl:c5duality}.
\end{proof}

\begin{cor}
$C_1$ is an independent set in $G$.
\end{cor}
\begin{proof}
Suppose for a contradiction there is an edge $\{u,u'\}$ with $u,u'\in C_1$.
There are at least \[\deg_{B_0}(u) \cdot \left(|A_0| - \eps' n\right) > \frac
n{25} \cdot \left(\frac n4 -23\eps'n\right) > \frac{n^2}{101}\] edges $\{v,w\}$
with $v \in N_{B_0}(u)$ and $w \in N_{A_0}(v)$.  However, the vertices $w$ and
$u'$ have a common neighbor in $B_0 \setminus \{v\}$ and hence
$\left|E(A_0,B_0)\cap \cC_5(G)\right| > n^2/101$; a contradiction with Corollary~\ref{cor:c5a0b0edge}.
\end{proof}

We define $C:=C_0 \cup C_1$,  and $D_1 : = Z \setminus (A_1 \cup B_1 \cup C_1)$.
By the definition of the sets $A$, $B_1$ and $C_1$, every vertex $v \in D_1$ has no neigbors in $A \cup B_0$. 
We now concentrate on the edges between $D_1$ and $C_0$.
\begin{claim}
\label{cl:c5degD1C0}
For every $v \in D_1$, $\deg_{C_0}(v) \ge n/25$.
\end{claim}
\begin{proof}
The vertex $v$ can be adjacent only to the vertices in $C_0 \cup D_0 \cup Z$,
and clearly every edge $\{v,w\}$ with $w \in C_0 \cup D_0$ occurs in $C_5$.
In particular, $v$ is incident to at most $22 \eps' n$ edges that do not occur in $C_5$.

As in Claim~\ref{cl:c5degC1B0}, if $\deg_{C_0}(v) < n/25$ then $\deg(v) < |D_0|+|Z|+n/25 < 0.396n$.
Therefore, removing the vertex $v$ and cloning the vertex $u'$ from Claim~\ref{cl:c5clone}
result in a graph contradicting Claim~\ref{cl:c5duality}.
\end{proof}

So the only possible edges that could be in $G$ but not following the pattern
of Construction~\ref{cstn:c5} are those between $B_1$ and $D_1$.  We rule them
out in the following claim.
\begin{claim}
$|E(B_1,D_1)| = 0$. 
\end{claim}
\begin{proof}
Suppose for contradiction there is an edge $\{u,v\}$ with $u \in B_1$ and $v
\in D_1$.  Since any vertex $x \in B_0$ has $\deg_{C_0}(x) > |C_0| - n/25$,
the vertices $v$ and $x$ have a common neighbor and hence
\[ \left|E(A_0,B_0) \cap \cC_5(G)\right| \ge \deg_{A_0}(u) \cdot (1-\eps')|B_0| > 0.03n^2,\]
which indeed contradicts Corollary~\ref{cor:c5a0b0edge}.
\end{proof}

Let $D := D_0 \cup D_1$. Putting everything together, we conclude that the edges in $G$ are as in
Construction~\ref{cstn:c5}.
\begin{cor}
$V(G) = A \cupdot B \cupdot C \cupdot D$ and 
$E(G) \subseteq E(A,B) \cupdot E(B,C) \cupdot E(C,D) \cupdot \binom{D}{2}$.
\end{cor}

In particular, the set of edges $E(G) = \cC_5(G) \cupdot E(A,B)$.
Since $G$ is minimizing $|\cC_5(H)|$ among all graphs in $H \in \cE'_n$, we immediately
conclude the following.
\begin{claim}
$|E(A,B)| = |A||B| = \widetilde{F}(n)$.
\end{claim}

Therefore, the quadruple $(|A|,|B|,|C|,|D|)$ is $n$-extremal which finishes the proof of the theorem.
\end{proof}


\section{Exact result for longer odd cycles}
\label{sec:c7+exact}

As in the previous section, for an integer $k\ge 3$ and a graph $G$ we define
$\cC_{2k+1}(G)$ to be the set of all edges of $G$ that occur in a copy of
$C_{2k+1}$ in $G$. In other words,
\[\cC_{2k+1}(G) := \bigcup\limits_{H \subseteq G, H \cong C_{2k+1}} E(H) .\]

Recall $\cE_n$ and $\cE'_n$ are the sets of all $n$-vertex graphs with exactly $\lfloor n^2/4\rfloor +1$ edges
and at least $\lfloor n^2/4\rfloor +1$ edges, respectively.
For any $k \ge 3$, let 
\[F_{2k+1}(n) := \min\limits_{G \in \cE_n}|\cC_{2k+1}(G)|,\]
and $\widetilde{F}_{2k+1}(n) := \lfloor n^2/4\rfloor +1 - F_{2k+1}(n)$.
Finally we define $\cG^{2k+1}_n \subseteq \cE'_n$ to be the set of all $G \in \cE'_n$ with $|\cC_{2k+1}(G)| = F_{2k+1}(n)$.
As we will show, for any $k\ge \ell\ge3$, there exists a sufficiently large $n_0:=n_0(k)$
such that $\cG^{2k+1}_n = \cG^{2\ell+1}_n$ for all $n \ge n_0$.

\begin{theorem}
\label{thm:c7+exact}
For any integer $k \ge 3$ there exists an integer $n_0$ such that the following holds for any $n \ge n_0$.
If $G \in \cG^{2k+1}_n$, then $V(G)$ can be partitioned 
into four sets $A$, $B$, $C$ and $D$ such that
\begin{itemize}
\item $|A|=\lfloor \frac{n-2}{6}\rfloor$, $|B|=\lfloor\frac{n+1}{6}\rfloor$, $|C|=1$ and $|D|=\lfloor \frac{2n+1}{3} \rfloor$.
\item $A$ and $B$ are independent sets of $G$,
\item $\{u,v\} \in E(G)$ for any $u \in A\cup C$ and $v \in B$,
\item $\{u,v\} \notin E(G)$ for any $u \in A$  and $v \in C \cup D$, and
\item $\{u,v\} \notin E(G)$ for any $u \in B$  and $v \in D$.
\end{itemize}
In particular, $F_{2k+1}(n) = \begin{cases}
2n^2/9+1   & \textrm{for } n \equiv 0 \mod 6,\\
2n^2/9+(n+13)/18   & \textrm{for } n \equiv 1 \mod 6,\\
2n^2/9-(n-22)/18   & \textrm{for } n \equiv 2 \mod 6,\\
2n^2/9+1   & \textrm{for } n \equiv 3 \mod 6,\\
2n^2/9+(n+22)/18   & \textrm{for } n \equiv 4 \mod 6,\\
2n^2/9-(n-13)/18   & \textrm{for } n \equiv 5 \mod 6.
\end{cases}$
\end{theorem}

\begin{proof}
Let $V:=V(G)$. 
For any $\eps' >0$ there is a choice of $\eps < \eps'$ and a large enough
constant $n_0 \in \Nat$ such that if $n \ge n_0$, then  
the stability result proven in Theorem~\ref{thm:c7+uniq} and the fact $G \in
\cG^{2k+1}_n$ imply that there exists a set of at most $\eps' \cdot n$ vertices
$C'$ such that the
subgraph $G[V\setminus C']$ is disconnected, and it contains a connected
component $D'$ with at least $(2/3 - \eps')\cdot n$ vertices and minimum degree at least $(2/3-\eps') \cdot n$.
This follows because, after removing at most $(\eps' / 2)\cdot n$ exceptionally behaving vertices $C'_1$,
any pair of two vertex-disjoint edges $e_1=x_1y_1$ and $e_2=x_2y_2$, where $x_1$ and $x_2$
are from the nearly clique part and $y_1$ and $y_2$ from the nearly complete bipartite part,
yields that at least $n/7$ edges incident to $y_i$, where $i\in\{1,2\}$, occur in some $C_{2k+1}$.
However, the nearly clique part contains at least $(4/9 - \eps'/15)\cdot n^2$ edges of $G$
and each such an edge occurs in some $C_{2k+1}$. Therefore, there must be less than $(\eps'/2) \cdot n$ vertices $C'_2$
such that the nearly clique part in $G$ forms a connected component of $G[V \setminus (C'_1 \cup C'_2)]$.
We conclude that there is a partition of $V$ into four parts $A'$, $B'$, $C'$ and $D'$ that satisfies
\begin{itemize}
\item $|A'|=(1/6 \pm \eps') \cdot n$, $|B'|=(1/6 \pm \eps') \cdot n$, $|C'| < \eps' \cdot n$, $|D'| = (2/3 \pm \eps') \cdot n$,
\item $\forall v\in A': \deg_{B'}(v) \ge (1/6-\eps') \cdot n$,
\item $\forall v\in B': \deg_{A'}(v) \ge (1/6-\eps') \cdot n$,
\item $\forall v\in D': \deg(v) \ge (2/3 - \eps') \cdot n$, and
\item $E(A' \cup B',D') = 0$,
\end{itemize}
Note that these properties yield that the induced subgraph $G[D']$ has edge-density at least $1-\eps'$, and $|E(A',B')| \ge |A'||B'| - 4\eps' \cdot n^2$. 

We start our exposition with a direct analogue of Claim~\ref{cl:c5duality}.
\begin{claim}
\label{cl:c7+duality}
There is no $n$-vertex graph $G' \in \cE'_n$ with $|E(G') \setminus \cC_{2k+1}(G')| > \widetilde{F}_{2k+1}(n)$.
\end{claim}
\begin{proof}
As otherwise removing from $G'$ arbitrarily chosen $|E(G')| - \lfloor n^2/4 \rfloor - 1$ edges in $\cC_{2k+1}(G')$
yields an $n$-vertex graph with less than $F_{2k+1}(n)$ edges that occur in $C_{2k+1}$, a contradiction.
\end{proof}

We continue by showing that both $A'$ and $B'$ are in fact independent sets in $G$.
\begin{claim}\label{cl:c7+ABneighbors}
No $v \in V$ is adjacent to $u \in A'$ and $w \in B'$.
\end{claim}
\begin{proof}
As otherwise, we will actually show that almost every edge of $G$ occurs in some $C_{2k+1}$.

Firstly, recall that all the edges of $G[D']$ occur in $C_{2k+1}$.
If there would be a vertex $v$ adjacent to $u \in A'$ and $w \in B'$, then we
can find a copy of $C_{2k+1}$ containing any given edge $\{u',w'\}$
with $u' \in A' \setminus \{u\}$ and $w' \in B' \setminus \{w\}$ in the
following way:  let $u_0 \in A'$ be an arbitrary common neighbor of $w$ and
$w'$, and let $P$ be a $(2k-3)$-vertex path between $u$ and $u'$ disjoint from
$w$, $w'$ and $u_0$.  Note that such a path exists because every vertex in $A$
has at least $|B| - 2\eps' n$ neighbors in $B$, and symmetrically every vertex
in $B$ has at least $|A| - 2\eps' n$ neighbors in $A$.                                  
Therefore, $vwu_0w'P$ is a copy of $C_{2k+1}$ in $G$ containing the edge $\{u',w'\}$.
It follows that $|E(G) \setminus \cC_{2k+1}(G)| \le \eps' n^2$,
which clearly contradicts the fact that $G \in \cG^{2k+1}_n$.
\end{proof}

\begin{cor}
$A'$ and $B'$ are independent sets in $G$.
\end{cor}

Let $C \subseteq V$ be a minimum-size set so that $G-C$ is disconnected and one
of its connected components is a bipartite graph $(A,B)$ with minimum degree at
least $n/7$.  Clearly, this is well defined because $C'$ 
has the bipartite graph $(A',B')$ as one of the components.
Moreover, among all such cuts $C$ of the minimum size, we choose such a $C$
that $|A|+|B|$ is as large as possible.

Let $D : = V \setminus (A \cup B \cup C)$. Because we already have a partition $(A',B',C',D')$ with $(A',B')$ being a bipartite graph and $D'$ of edge density at least $1-\eps'$, 
one can easily see that the partition $A \cupdot B \cupdot C \cupdot D$ of $V$ 
behaves very similarly to the original partition $A' \cupdot B' \cupdot C' \cupdot D'$. 
In particular,
\begin{itemize}
\item $|A|=(1/6 \pm 2\eps')\cdot n$,
\item $|B|=(1/6 \pm 2\eps')\cdot n$,
\item $|C| < \eps' \cdot n$,
\item $|D \setminus D'| < \eps' \cdot n$, and
\item $E(D) \subseteq \cC_{2k+1}(G)$.
\end{itemize}
Following the proof of Claim \ref{cl:c7+ABneighbors} we get also that there is no vertex $v \in V$ adjacent to $u \in A$ and $w \in B$.

Now let use an argument analogous to the one used in Section~\ref{sec:c5exact}
to show that $G$ must have a large minimum degree.
\begin{claim}
\label{cl:c7+clone}
There is a vertex $v \in A$ that is incident to at least $n/6 - 9\eps' \cdot n$ edges not in $\cC_{2k+1}(G)$.
\end{claim}
\begin{proof}
Suppose not, then the number of edges not in $\cC_{2k+1}(G)$ is smaller than 
\[|A| \cdot \left(\frac n6 - 9\eps' \cdot n\right) +
  |C| \cdot n \le \left(\frac n6 + 2\eps' \cdot n\right) \cdot \left(\frac n6 - 9\eps' \cdot n\right)
  + \eps' \cdot n^2 < n^2/36 - \eps'/6 \cdot n^2.
\]
Therefore, there are more than $2n^2/9 + \eps' n^2/6$ edges in $\cC_{2k+1}(G)$,
contradicting the extremality of $G$ since Construction~\ref{cstn:cliquebip}
has at most $2n^2/9 + (n+22)/18$ edges that occur in $C_{2k+1}$.
\end{proof}

\begin{cor}
\label{cor:c7+mindeg}
For any $v \in V$, $\deg(v) > n/6 - 9\eps' \cdot n$.
\end{cor}

\begin{proof}
If there exist a vertex $w\in V$ of a smaller degree, 
then by removing $w$ and adding a clone of the vertex $v$ from the above claim, 
we improve the graph contradicting Claim~\ref{cl:c7+duality}.
\end{proof}


\begin{claim}
Every $u,w\in D$ have a common neighbor in $D$.
\end{claim}
\begin{proof}
First observe that all pairs of vertices $u \in D'$ and $w \in D$ have a common neighbor in $D'$.
Suppose for a contradiction there exist two vertices $u,w \in D\setminus D'$ with no common neighbor in $D$.
Then consider the graph $G'$ obtained from $G$ by removing both $u$ and $w$, adding a new vertex $u'$ connected to the
whole set $D' \cap D$, and adding a new vertex $w'$ which will be a clone of the vertex $v$ from Claim~\ref{cl:c7+clone}.

We removed at most $|D| + 2\eps' n$ edges from $G$, and added 
$\deg(u') + \deg(w') \ge  |D| + n/6 - 10\eps' n$ new edges.
Moreover, all the removed edges were in $\cC_{2k+1}(G)$, so $G'$ contradicts Claim~\ref{cl:c7+duality}.
\end{proof}

Now let us concentrate on the vertex-cut $C$.
Firstly, we observe that $C$ must be non-empty.
\begin{claim}
$G$ is a connected graph. In particular $|C| \ge 1$.
\end{claim}
\begin{proof}
If $G$ is disconnected, take any two connected components of $G$ and add
one edge between them. Clearly, the added edge does not occur in any cycle
contradicting $G \in \cG^{2k+1}_n$.
\end{proof}

In the following series of claims, we will show that $|C| \le 1$. In order to
do so, we split the vertices of $C$ based on their adjacencies to $A$ and $B$
(recall no vertex can be adjacent to both $u \in A$ and $w \in B$).
Let $C_A := \{v \in C \cond \deg_A(v) > 0 \}$ and $C_B := C \setminus C_A = \{v \in C \cond \deg_B(v) > 0\}$.

\begin{claim}\label{cl:c7+Cbipartite}
$|E(C_A)| = |E(C_B)|=0$.
\end{claim}
\begin{proof}
Suppose the claim is false. Without loss of generality, there is an edge $\{v_1,v_2\} \in E(C_A)$.
Consider any two vertices $u_1 \in N_A(v_1)$ and $u_2 \in N_A(v_2)$, any vertex $w_1 \in N_B(u_1)$,
any vertex $u_3 \in N_A(w_1)\setminus \{u_1,u_2\}$, and a $(2k-3)$-vertex path $P$ between the vertices $u_2$ and $u_3$
with the internal vertices disjoint from $u_1$, $v_1$, $v_2$ and $w_1$. It follows that $v_1u_1w_1Pv_2$ yields a copy of $C_{2k+1}$ in $G$.
Therefore,
\[|E(A,B) \cap \cC_{2k+1}(G)| \ge \deg_B(u_1) \cdot \left(\deg_A(w_1)-2\right) > n^2/50, \]
and hence $|\cC_{2k+1}(G)| > 2n^2/9 + (n+22)/18$; a contradiction. 
\end{proof}

Next, we study the edges between the sets $C$ and $D$.
\begin{claim}
For any set $X \subseteq C$, $|N_D(X)| > |X|$.
In particular, every vertex $v \in C$ have at least two neighbors in $D$.
\end{claim}
\begin{proof}
Suppose for contradiction that there exists $X \subseteq C$ with $|N_D(X)| \le |X|$,
and let $Y:=N_D(X)$.
By Corollary~\ref{cor:c7+mindeg}, $\deg_{A \cup B}(v) > n/6 - 9\eps' n > n/7$
for any $v \in X$.
Therefore, $(C \cup Y) \setminus X$ is a vertex-cut of size at most
$|C|$ and $G[A \cup B \cup X]$ is a bipartite graph (from Claim \ref{cl:c7+Cbipartite}) with a minimum degree at least $n/7$
contradicting the choice of $C$.
\end{proof}

Since every $v \in C$ has at least two neighbors in $D$, we conclude that every edge between
$C$ and $D$ occurs in some $(2k+1)$-cycle, i.e., $E(C,D) \subseteq \cC_{2k+1}(G)$.




\begin{claim}
\label{cl:c7+noC_A-D-C_Bedge}
$|N_D(u_a) \cap N_D(u_b)| = 0$ for any $u_a \in C_A$ and $u_b \in C_B$.
\end{claim}
\begin{proof}
Suppose there exists $w \in N_D(u_a) \cap N_D(u_b)$.
Let $v_a \in N_A(u_a)$ and $v_b \in N_B(u_b)$ be chosen arbitrarily,
and consider the bipartite subgraph $(A',B')$ with $A' := N_A(v_b)$ and $B' := N_B(v_a)$.
It follows that $|E(A,B) \setminus E(A',B')| < 4\eps' n^2$. On the other hand,
any edge $\{x,y\} \in E(A',B')$ occurs in $C_{2k+1}$ for all $k\ge3$, a contradiction.
\end{proof}

\begin{claim}
$|C_A| \le 1$ and $|C_B| \le 1$.
\end{claim}
\begin{proof}
By symmetry, it is enough to prove that $|C_A| \le 1$. Suppose for contradiction that $|C_A|\ge 2$.

We first consider the case when there are two vertices $u_1, u_2 \in C_A$ and
an edge $\{x_1,x_2\} \in E(D)$ with $x_1 \in N_D(u_1)$ and $x_2 \in N_D(u_2)$.
In other words, there is a $4$-vertex path with both of its endpoints in $C_A$.
Let $W := N_A\left(\{u_1,u_2\}\right)$. Note that $|W|\ge2$ as otherwise $(C
\cup W) \setminus \{u_1,u_2\}$ contradicts the minimality of $C$.
Since any two vertices $w_1,w_2 \in A$ have more than $2n/7 - |B| > 4|B|/7$
common neighbors in $B$, we conclude that $|E(W,B)\setminus \cC_{2k+1}(G)| < 3|B|/7
< n/13$. Also, $\deg(w) \le |B| + |C| < n/5$ for any $w \in A$. It follows that
the graph obtained from $G$ by removing the vertex-set $W$, adding $|W|-1$ new
vertices fully connected to $D$, and adding a clone of a~vertex~$v$ from
Claim~\ref{cl:c7+clone} yields a graph $G'$ with more than $n^2/4$ edges
and $|E(G') \setminus \cC_{2k+1}(G')| > |E(G) \setminus \cC_{2k+1}(G)|$, a contradiction
with Claim~\ref{cl:c7+duality}. 

For the rest of the proof, we may assume there is no $4$-vertex path with the
endpoints in $C_A$. In particular, at most one vertex from $C_A$ can have
$\Omega(\eps n)$ neighbors in $D$.
Let us now focus on the edges between $C_A$ and $A$ that
are not in $\cC_{2k+1}(G)$. Clearly, there are at most $|A|$ of them since any two
edges $e_1,e_2 \in E(C_A,A)$ with $e_1 \cap e_2 \in A$ occur in $C_{2k+1}$.  Now
suppose there exist two vertices $u_1,u_2 \in C_A$ that both have less than
$n/24$ neighbors in $D$.  By Corollary~\ref{cor:c7+mindeg}, it follows that
$|N_A(u_1) \cap N_A(u_2)| > |A|/3$. On the other hand, \hbox{$\deg(u_1) + \deg(u_2) < n/2$}.  
Therefore, replacing the vertices $u_1$ and $u_2$ with one new vertex adjacent
to every vertex in $D$ and a clone of the vertex $v$ from
Claim~\ref{cl:c7+clone} again yields a contradiction with
Claim~\ref{cl:c7+duality}.

We conclude that if $|C_A| \ge 2$, then $C_A = \left\{u_1, u_2\right\}$ and $\deg_D(u_1) \ge n/24$.
Note that $u_1$ or $u_2$ is incident to at most $|A|/2$ edges that do not occur
in $C_{2k+1}$.  Let $u \in C_A$ be this vertex and let $u' \in C_A$ be the
other vertex. Since $\deg_D(u_2) \ge 2$ and hence $\deg_D(u') \ge 2$, there
are at least $2 \cdot \left(\deg_D(u_1) - 2\right) > \deg_D(u)$ 
non-edges in $D$ between $N_D(u)$ and $N_D(u')$.
Therefore, removing the vertex $u$, adding all the edges $\{w,w'\}$ with
$w \in N_D(u)$ and $w' \in N_D(u')$, and adding a clone of the vertex $v$ from
Claim~\ref{cl:c7+clone} contradicts Claim~\ref{cl:c7+duality}, which finishes
the proof of the claim.
\end{proof}

It remains to show that we cannot have both $|C_A|=1$ and $|C_B|=1$.
\begin{claim}
$|C| = 1$.
\end{claim}
\begin{proof}
Suppose for contradiction there are vertices $u_a \in C_A$ and $u_b \in C_B$.
Firstly, recall that $N_D(u_a) \cap N_D(u_b) = \emptyset$ by Claim~\ref{cl:c7+noC_A-D-C_Bedge}.

Now let us prove that both $|N_A(u_a)|$ and $|N_B(u_b)|$ must have quite small sizes, say less~than~$n/24$.
Suppose, without loss of generality, that $|N_A(u_a)| \ge n/24$.
Our aim now is to show that any edge incident to $u_b$ is contained in some $C_{2k+1}$.
Consider any vertex $v \in N_B(u_b)$. Since $\deg_A(v) > |A| - 11\eps' n$, the vertices $v$
and $u_a$ have a common neighbor $w \in A$. Therefore, $u_awvu_bP$ gives a
copy of $C_{2k+1}$, where  $P$ is a $(2k-3)$-vertex path in $D$ between $x \in N_D(u_a)$ and $x' \in N_D(u_b) \setminus \{x\}$.
We conclude that every edge incident to $u_b$ is in $\cC_{2k+1}(G)$. Moreover, there is at least one edge in $E(A,B) \cap \cC_{2k+1}(G)$.
But then consider a graph $G'$ obtained from $G$ by removing at most $|B|$ edges between $u_b$ and $B$, and adding at least $|D| > |B|$
missing edges between $C$ and $D$. Since no edge from $E(A,B)$ is in $\cC_{2k+1}(G')$, the graph $G'$ contradicts Claim~\ref{cl:c7+duality}.

It remains to consider the case when both $|N_A(u_a)|$ and $|N_B(u_b)|$ have sizes less~than~$n/24$.
But then removing all the edges from, say, $u_a$ to $A$, and adding all the missing edges between $u_b$ and $B$ yield
a graph $G'$ that again contradicts Claim~\ref{cl:c7+duality}.
\end{proof}

This gives us a complete information on the structure of the extremal graphs. 

\begin{cor}
$E(G) \subseteq E(A,B) \cupdot E(B,C) \cupdot E(C,D) \cupdot \binom{D}{2}$. 
\end{cor}

It follows that all the edges that do not occur in $C_{2k+1}$ are incident to
vertices in $B$.
\begin{cor}
$E(A,B)=|A||B|$ and $|B|\ge|A|$.
Moreover, $F_{2k+1}(n) = \lfloor\frac{n^2}{4}\rfloor +1 - (|A|+1)|B|$.
\end{cor}

Finally, knowing the structure, it is straightforward to get the fact that $G \in \cG^{2k+1}_n$ yields
that $|A| = \lfloor(n-2)/6\rfloor$, $|B| = \lfloor(n+1)/6\rfloor$ and $|D| = \lfloor(2n+1)/3\rfloor$.
\end{proof}



\section{Concluding remarks}
\label{sec:remarks}

For an $n$-vertex graph $G$ with $\lfloor n^2/4\rfloor + 1$ edges, we
determined the asymptotic minimum number of the edges of $G$ that occur in some
copy of $C_5$ in $G$, and for any $k\ge3$, the exact minimum number of the
edges that occur in $C_{2k+1}$.
Our results show that the pentagon case has a very different behavior compared
to all the longer odd cycles. These results confirm a conjecture of
F\"uredi and Maleki, who proved the optimal asymptotic bounds under a stronger
assumption that $G$ has $(1/4 + \eps)n^2$ edges.

Our main tool was an application of techniques from finite forcibility in the
setting of flag algebras, combined with stability results on triangle-free
graphs.  This was crucial for dealing with $n$-vertex graphs that have only
$\lfloor n^2/4\rfloor + 1$ edges. 
We believe that our approach can be adapted to various other scenarios, and we
intend to investigate this direction further.

We were also able to guide flag algebras to give us additional structural
information for extremal configurations which yielded the corresponding
stability results. These stability results allowed us to fully describe the
structure of all the sufficiently large tight constructions. 

If $G$ contains $\alpha n^2$ edges for some $\alpha > 1/4$, then a standard averaging
argument yields that $G$ must contain much more edges that occur in $C_{2k+1}$, for $k$ being fixed,
than Theorems~\ref{thm:c5} and~\ref{thm:c7+} guarantee for $\lfloor n^2/4\rfloor + 1$ edges.
However, the averaging argument yields only a weak improvement. 
F\"uredi and Maleki~\cite{bib:FurMal} determined an asymptotically optimal
lower bound for this problem.  Note that the corresponding approximate result
for triangles was proven by F\"uredi and Maleki in~\cite{bib:FurMalTria}. 
%


F\"uredi and Maleki~\cite{bib:FurMal} also considered a more general question,
where instead of minimizing the number of edges that occur in odd cycles of a fixed length,
one minimizes the number of edges that occur in copies of $F$ for some fixed graph $F$.
If the graph $F$ has chromatic number $\chi = 3$, they obtained an asymptotically tight solution to this question.
However, for graphs $F$ with chromatic number $\chi \ge 4$, these questions are widely open.



\medskip

{\noindent \bf Acknowledgments.}

The authors thank Zoltan F\"uredi and Zeinab Maleki for discussing the results of~\cite{bib:FurMal}
and the relation to the results obtained in this paper, and to Shoham Letzter for her suggestions
regarding the results presented in Sections~\ref{sec:c5exact} and~\ref{sec:c7+exact}.
We also thank the anonymous referees for their valuable comments, and Jake
Cooper and Dan Kr\'al' for fruitful discussions at the beginning of this
project. All of these greatly improved the presentation of our results.


\begingroup
\bibliographystyle{abbrv}
\linespread{0.97}\selectfont
\bibliography{refs}
\endgroup
\newpage


\appendix

\section{Formal verification of correctness of Propositions~\ref{prop:c5flag} and~\ref{prop:c7+flag}}
\label{apx:verify}

In order to verify the correctness of the claimed identities, given 
a~\emph{proof-certificate} consisting of~matrices $\widehat{L}$, $M_\lambda$, $\widehat{B}$, $M_\beta$, $\widehat{R}$, $M_\rho$ and two numbers $a,b > 0$,
we perform the following 8 steps:
\begin{itemize}
\item[1)] Generate all the non-isomorphic flags in the sets $\cH_6$, $\cH_4^\lambda$, $\cH_4^\beta$, and $\cH_4^\rho$\,,

\item[2)] For every $F_1^\sigma,F_2^\sigma \in \cH_4^\sigma$, where $\sigma \in \{\lambda,\beta,\rho\}$,
express $\unlab{F^\sigma_1 \times F^\sigma_2}{\sigma}$\, as $\displaystyle\sum\limits_{H \in \cH_6} p_H^{F_1^\sigma,F_2^\sigma} \cdot H$\,,
 
\item[3)] Verify that the three matrices $\widehat{L}$, $\widehat{B}$, and $\widehat{R}$ are positive definite,

\item[4)] Express
$ \unlab{v_\lambda^\T M_\lambda^\T \cdot \widehat{L} \cdot M_\lambda v_\lambda}\lambda
+ \unlab{v_\beta^\T M_\beta^\T \cdot \widehat{B} \cdot M_\beta v_\beta}\beta
+ \unlab{v_\rho^\T M_\rho^\T \cdot \widehat{R} \cdot M_\rho v_\rho}\rho
$\,
as $\displaystyle\sum\limits_{H \in \cH_6} \zeta_H \cdot H$\,,

\item[5)] Express
$\left(\figR \,+\, \figB \,-\, \figN  \right) \times \figEMP$\, as $\displaystyle\sum\limits_{H \in \cH_6} \gamma_H \cdot H$\,
and
$\left(\figR \,+\, \figB \,-\, \figN  \right) \times \figCOCh$\, as $\displaystyle\sum\limits_{H \in \cH_6} \gamma'_H \cdot H$\,,

\item[6a)] In the case of Proposition~\ref{prop:c5flag}, express
$\figTRG \,\times\,\left(8\cdot\figR - \left({2+\sqrt{2}}\right) \cdot \figV \, \right)$\, as $\displaystyle\sum\limits_{H \in \cH_6} \kappa_H \cdot H$\,,
\item[6b)] In the case of Proposition~\ref{prop:c7+flag}, express
$\figTRG \,\times\, \left( 9\cdot\figR - 4 \cdot \figV \, \right)$\, as $\displaystyle\sum\limits_{H \in \cH_6} \kappa_H \cdot H$\,,

\item[7)] For every $H \in \cH_6$, verify that
\begin{equation}\label{eq:finalcoef}
\kappa_H  \ge 
\zeta_H + a \cdot \gamma_H + b \cdot \gamma'_H  \,,
\end{equation}

\item[8a)] In the case of~Proposition~\ref{prop:c5flag}, verify that the inequality (\ref{eq:finalcoef}) is strict for every $H \in  \cP_5 \cup \{C_4^X\}$, and 
\item[8b)] In the case of~Proposition~\ref{prop:c7+flag}, verify that the inequality (\ref{eq:finalcoef}) is strict for every $H \in \cP_4$.
\end{itemize}

In our verification scripts on the webpage \mbox{\linktocalculations}, we implement the first and the second step by a
simple exhaustive search over all the possibilities.  The positive-definiteness
of the given matrices is verified by finding their $LDL^\T$ decompositions and testing whether
all diagonal entries of $D$ are positive. 
Note that for the matrix decomposition, we use the corresponding function in SAGE which uses exact arithmetics.

Next, given a proof-certificate
$\left(\widehat{L},M_\lambda,\widehat{B},M_\beta,\widehat{R},M_\rho,a,b\right)$
and the values of $p^{F^\sigma_1,F^\sigma_2}_H$ computed in the second step, we
directly compute the values of $\zeta_H$, again using exact arithmetics implemented in SAGE.
Note that to do so, we only need to perform summation and multiplication in $\Rat\left[\sqrt 2\right]$.

To find the values $\zeta_H$, $\gamma_H$, $\gamma_H'$ and $\kappa_H$ from the steps 4-6,
we again go exhaustively through all the possibilities. Finally, the steps 7 and 8
are verified by a direct computation in $\Rat\left[\sqrt 2\right]$.

\end{document}